\documentclass[a4paper,11pt]{article}

\usepackage{mathrsfs}
\usepackage{amsmath}
\usepackage{bm}
\usepackage{amssymb}
\usepackage{makeidx}
\makeindex
\usepackage[dvipsnames]{xcolor}
\usepackage[colorlinks=true, linkcolor=blue, citecolor=blue, urlcolor=blue, hypertexnames=false]{hyperref}
\usepackage{xurl}
\usepackage{needspace}
\usepackage{titlesec}
\usepackage{amsthm}
\usepackage{setspace}
\usepackage{pifont}
\usepackage{esint}
\usepackage{graphicx}
\usepackage{float}
\usepackage{enumerate}
\usepackage{tikz}

\numberwithin{equation}{section}
\usepackage[backend=biber,style=alphabetic,sorting=nyt,giveninits=true]{biblatex}
\DeclareNameAlias{default}{given-family}
\DeclareNameAlias{sortname}{family-given}

\DeclareFieldFormat[article]{title}{#1}

\renewbibmacro{in:}{}

\DeclareFieldFormat{pages}{#1}

\DeclareFieldFormat[article]{journaltitle}{\mkbibemph{#1}}

\DeclareFieldFormat[article]{volume}{#1}

\renewbibmacro*{journal+issuetitle}{%
  \usebibmacro{journal}%
  \setunit*{\addspace}%
  \printfield{volume}%
  \setunit{\addspace}%
  \printtext[parens]{\printfield{year}}%
  \iffieldundef{number}
    {}
    {\setunit{\addcomma\space}\printtext{no.\space\printfield{number}}}%
  \newunit
}

\renewbibmacro*{issue+date}{}
\titleformat{\section}[block]
  {\bfseries\Large}
  {\thesection}
  {1em}
  {}

\titleformat{\subsection}[block]
  {\bfseries\large}
  {\thesubsection}
  {1em}
  {}

\titleformat{\subsubsection}[runin]
  {\bfseries\normalsize}
  {\thesubsubsection}
  {0.75em}
  {}

\titlespacing*{\section}{0pt}{3.5ex plus 1ex minus .2ex}{2.3ex plus .2ex}
\titlespacing*{\subsection}{0pt}{3ex plus 1ex minus .2ex}{1.5ex plus .2ex}
\titlespacing*{\subsubsection}{0pt}{2ex plus 0.5ex minus .2ex}{1em}

\titleformat{\section}[block]{\bfseries\Large}{\thesection}{1em}{}

\newtheoremstyle{mystyle}
{}{}{\upshape}{}{\bfseries}{.}{.5em}{}
\theoremstyle{mystyle}

\theoremstyle{definition}
\newtheorem{theorem}{Theorem}[section]
\newtheorem{lemma}[theorem]{Lemma}
\newtheorem{proposition}[theorem]{Proposition}

\newtheorem{coro}[theorem]{Corollary}
\newtheorem{definition}[theorem]{Definition}

\newtheorem{remark}[theorem]{Remark}

\begin{document}
\nocite{*}

	\title{\large
		Nonlocal Characterizations of Stochastic Completeness \\ on Complete Riemannian Manifolds}
	
	\author{\small Rui Chen, Bobo Hua}
	\date{}
	\maketitle
	\thispagestyle{empty}
	\pagenumbering{arabic}

\noindent\textbf{Abstract:}
In this paper, we first prove a generalized conservation principle for the
subordinate heat semigroup on complete Riemannian manifolds: for every
\(0<s<1\) and \(t>0\),
\[
T_t^{(s)}\mathbf 1+\int_0^t T_\tau^{(s)}\mathcal R_s\,d\tau=1
\qquad\text{on }M,
\]
where \(\mathcal R_s\) is the intrinsic killing term measuring the loss of mass
of the subordinate semigroup, and the condition \(\mathcal R_s\equiv0\) is
equivalent to the stochastic completeness of \(M\).

We then provide several new nonlocal characterizations of stochastic
completeness. In particular, we show that stochastic completeness is equivalent
to genuinely nonlocal conditions, including the zero-mean identity
\[
\int_M (-\Delta)^s\varphi\,dV_g=0
\qquad\forall\,\varphi\in C_c^\infty(M),
\]
as well as the uniqueness of bounded distributional solutions to the associated
fractional elliptic and parabolic equations. We also revisit the equivalent
\(L^1\)-core characterization for the generator of the heat semigroup, which
plays an important role in our approach.

In addition, we prove \(L^p\)-contractivity and smoothing properties of the
subordinate semigroup, establish both short-time and long-time asymptotic
results for the fractional heat kernel, derive the short-time asymptotics of
jump probabilities for the associated Markov process, and study the variational
characterization and minimality properties of the fractional resolvent.
Together, these results provide a unified analytic and probabilistic framework
for conservation, killing, and nonlocal diffusion phenomena on complete
Riemannian manifolds.



	\bigskip

	% ========== Table of contents ==========
	\tableofcontents
	\thispagestyle{empty}

	% ========== Main text ==========
	\setcounter{page}{1}
	
\section{Introduction and Main Results}\label{Introduction}
	In this paper, we establish some nonlocal characterizations of stochastic completeness on complete Riemannian manifolds in terms of the fractional Laplacian. Our results reveal both a close relation and a fundamental distinction between the local and nonlocal settings.

    Stochastic completeness is a fundamental global property in geometric analysis and probability theory on manifolds. It expresses the conservation of total heat mass under the heat flow, or equivalently, the non-explosion of the Brownian motion associated with the Laplace--Beltrami operator. More precisely, if \(\{P_t\}_{t\ge0}\) denotes the heat semigroup on a complete Riemannian manifold \((M,g)\), then stochastic completeness means
\[
P_t\mathbf 1=\mathbf 1
\qquad\text{for all }t>0.
\]
This property plays a central role in the study of heat kernels, global geometry, potential theory, and the interplay between analytic and probabilistic aspects of diffusion processes on manifolds. There are several classical characterizations of stochastic completeness. A basic analytic characterization is the following theorem.

\begin{theorem}\cite[Theorem 8.18 and Remark 8.19]{grigoryan2009heat}\label{thm:classical-stochastic-completeness}
 Fix \(\alpha>0\) and \(T\in(0,\infty]\). Then the following assertions are equivalent:

\medskip
\noindent\textnormal{(i)} \((M,g)\) is stochastically complete, namely $P_t\mathbf 1=\mathbf 1$ for all $t>0.$

\medskip
\noindent\textnormal{(ii)} The equation $\Delta v=\alpha v$ in $M$
admits no nontrivial bounded nonnegative solution.

\medskip
\noindent\textnormal{(iii)} The only bounded nonnegative solution of
$-\Delta v+\alpha v=1$ in $M$ is $v=\alpha^{-1}$.

\medskip
\noindent\textnormal{(iv)} The Cauchy problem
\[
\begin{cases}
\partial_t u=\Delta u & \text{in }(0,T)\times M,\\
u(\cdot,0)=u_0 & \text{on }M
\end{cases}
\]
has at most one bounded solution for every bounded initial datum \(u_0\).
\end{theorem}

This classical equivalence goes back to Grigor'yan; see, for instance,
\cite[Theorem 8.18]{grigoryan2009heat} and the earlier papers
\cite{grigor1986stochastically,grigor1989stochastically}. More broadly, the study of stochastic completeness originates in classical works of Gaffney \cite{Gaffney1959}, Khas'minskii \cite{khas1960ergodic}, Yau \cite{Yau1978}, and Azencott \cite{azencott1974behavior}. A major milestone was Grigor'yan's seminal discovery that, on a geodesically complete manifold, stochastic completeness is closely related to the volume growth of geodesic balls; see \cite{grigor1986stochastically} and \cite[Theorem 1.6]{grigor2011stochastic}. Closely related characterizations also arise from weak maximum principles at infinity and Khas'minskii-type criteria, which connect stochastic completeness with the existence of suitable exhaustion functions and superharmonic barriers; see, for example, Khas'minskii's foundational work \cite{khas1960ergodic} and the later treatments by Pigola, Rigoli, and Setti \cite{pigola2003remark,pigola2008vanishing}. Altogether, these results show that stochastic completeness is deeply tied to the geometry at infinity, and in particular to the balance between volume growth and the tendency of heat to escape to infinity. Related work on stochastic completeness of jump processes and the Green formula appears in \cite{GrigoryanHuangMasamune2012,GrigoryanMasamune2013}.

From this perspective, the classical local theory exhibits a remarkable limitation: for the Laplace--Beltrami operator, the identity
\[
\int_M \Delta\varphi\,dV_g=0
\qquad\forall\,\varphi\in C_c^\infty(M)
\]
always holds by the divergence theorem, independently of whether the manifold is stochastically complete. Therefore, compactly supported test functions cannot directly detect conservation in the local setting. By contrast, for the fractional Laplacian, the quantity
\[
\int_M (-\Delta)^s\varphi\,dV_g
\]
provides a sharp characterization of stochastic completeness. On the other hand, we show that on stochastically incomplete manifolds the loss of mass admits an explicit description in terms of the intrinsic killing term \(\mathcal R_s\), leading to the compensated conservation law
\[
T_t^{(s)}\mathbf 1+\int_0^t T_\tau^{(s)}\mathcal R_s\,d\tau=1
\qquad\text{on }M.
\]
Thus, although the fractional semigroup need not be conservative in the usual sense, its mass defect is precisely encoded by \(\mathcal R_s\).

In this paper, we always assume that \((M,g)\) is a connected, complete Riemannian manifold without boundary. In this setting, the Laplace--Beltrami operator \(-\Delta\), initially defined on \(C_c^\infty(M)\), admits a unique self-adjoint realization on \(L^2(M)\); for more details, see Section~\ref{comple}. Most of the main results can also be formulated on incomplete Riemannian manifolds, provided one works with the realization of the Laplacian and the corresponding heat semigroup associated with the minimal heat kernel, see \cite{grigoryan2009heat}. By the spectral theorem, there exists a unique projection-valued measure \(E\) such that $-\Delta=\int_{[0,\infty)}\lambda\,dE(\lambda).$

\smallskip

We now recall some natural ways to define the fractional Laplacian on \((M,g)\); see \cite{chen2025logarithmic}.

\medskip
\noindent\textbf{\textnormal{(i) Spectral definition.}}
Let \(s>0\). The fractional power \((-\Delta)^s\) is defined by functional calculus as
\[
(-\Delta)^s
:=
\int_{[0,\infty)} \lambda^s\,dE(\lambda).
\]

\medskip
\noindent\textbf{\textnormal{(ii) Bochner integral representation.}}
Let \(0<s<1\). Since \(-\Delta\) is nonnegative and self-adjoint, it generates the heat semigroup \(\{e^{t\Delta}\}_{t\ge0}\) on \(L^2(M)\). For every \(f\in \operatorname{Dom}((-\Delta)^s)\), one has
\[
(-\Delta)^s f
=
\frac{s}{\Gamma(1-s)}
\int_0^\infty \frac{f-e^{t\Delta}f}{t^{1+s}}\,dt,
\]
where the integral is understood as a strong improper Bochner integral in \(L^2(M)\), that is, the norm limit of the integrals over \([\varepsilon,R]\) as \(\varepsilon\downarrow0\) and \(R\to\infty\).

A natural question is whether \((-\Delta)^s\) admits a singular integral representation analogous to the Euclidean one. On general complete manifolds, however, this issue must be handled with particular care, since the validity of such a representation is closely related to the global behavior of the heat flow at infinity. To make this precise, for each \(t>0\) and \(x\in M\), we define the total mass function
\[
\Theta(t,x):=e^{t\Delta}\mathbf 1(x)=\int_M p(t,x,y)\,dV_g(y).
\]
Since the heat semigroup is positivity preserving and sub-Markovian, see Section \ref{comple}, one has
\[
0\le \Theta(t,x)\le 1,
\qquad t>0,\ x\in M.
\]
Moreover, by \cite[Section 3.2]{chen2025logarithmic}, the limit $\Theta_\infty(x):=\lim_{t\to\infty}\Theta(t,x)$
exists and belongs to \([0,1]\). In general, the function \(\Theta_\infty\) measures the possible loss of mass at infinity. In probabilistic terms, it coincides with the survival probability of Brownian motion, namely the probability that the diffusion starting from \(x\) does not explode in finite time. 

For \(f\in C_c^\infty(M)\), we introduce the heat-kernel fractional Laplacian
\[
(-\Delta)_{\mathrm{hk}}^s f(x)
:=
\operatorname{p.v.}\!\int_M \bigl(f(x)-f(y)\bigr)K_s(x,y)\,dV_g(y),
\]
with the principal value taken over $M\setminus B_\varepsilon(x)$ as
$\varepsilon\downarrow0$, where
\begin{equation}\label{frackernel}
K_s(x,y)
:=
\frac{s}{\Gamma(1-s)}
\int_0^\infty p(t,x,y)\,\frac{dt}{t^{1+s}},
\qquad x\neq y.
\end{equation}
According to the results proved in Section~\ref{Equiva}, this quantity is well defined in the principal-value sense, and
\begin{equation}\label{two equivalent}
    (-\Delta)^s f(x)
=
(-\Delta)_{\mathrm{hk}}^s f(x)+\mathcal R_s(x)\,f(x),
\end{equation}
where we denote
\begin{equation}\label{remainder}
    \mathcal R_s(x)
:=
\frac{s}{\Gamma(1-s)}
\int_0^\infty \bigl(1-\Theta(t,x)\bigr)\,\frac{dt}{t^{1+s}}.
\end{equation}
In particular,  the identity
\((-\Delta)^s f(x)=(-\Delta)_{\mathrm{hk}}^s f(x)\)
holds for every $f\in C_c^\infty(M)$ and $x\in M$ if and only if
$\Theta(t,x)\equiv1$, namely, when the manifold is stochastically complete.

To study the fractional Laplacian from both analytic and probabilistic viewpoints, we shall work with the subordinate semigroup
\(\{T_t^{(s)}\}_{t\ge0}=\{e^{-t(-\Delta)^s}\}_{t\ge0}\)
and the associated resolvent family
\[
R_\alpha^{(s)}:=\bigl((-\Delta)^s+\alpha\bigr)^{-1},
\qquad \alpha>0.
\]
These objects provide the natural nonlocal counterparts of the heat semigroup and the classical resolvent of the Laplace--Beltrami operator; see Section~\ref{subordinate} and Section~\ref{resolvent} for more details.

 Motivated by the decomposition \eqref{two equivalent}, we are led to expect that, on a stochastically incomplete manifold, the loss of mass of the fractional semigroup is precisely captured by the term \(\mathcal R_s\). In other words, although the fractional semigroup is not conservative in the usual sense, its mass defect should be exactly compensated by the contribution of \(\mathcal R_s\). This naturally leads to compensated conservation identities for both the fractional semigroup and its resolvent, in close analogy with the generalized conservation property for Schrödinger semigroups; see \cite{MasamuneSchmidt2020}.

\begin{theorem}\label{thm:generalized identity}
  Let \(0<s<1\). Then, for every \(t>0\),
\[
T_t^{(s)}\mathbf 1+\int_0^t T_\tau^{(s)}\mathcal R_s\,d\tau=1
\qquad\text{on }M,
\]
and, for every \(\alpha>0\),
\[
\alpha R_\alpha^{(s)}\mathbf 1+R_\alpha^{(s)}\mathcal R_s=1
\qquad\text{on }M.
\]  
\end{theorem}
Here the semigroup and resolvent actions on nonnegative measurable
functions are understood through their positive kernels.

We first present two natural characterizations of stochastic completeness in terms of the fractional semigroup and the associated resolvent. From the subordination formula, these equivalences are quite natural. 

\begin{theorem}\label{prop:conservativeness-subordination}
Let \((M,g)\) be a complete Riemannian manifold, let \(0<s<1\). Then, for every \(t>0\), the following statements are equivalent:
\begin{enumerate}
\item[(i)] For every $t>0$,
\(P_t\mathbf 1=\mathbf 1.\)
\item[(ii)]For every $t>0$,
\(T_t^{(s)}\mathbf 1=\mathbf 1.\)

\item[(iii)] For every $t>0$ and $\varphi\in L^1(M),$
\[\int_M T_t^{(s)}\varphi\,dV_g
=
\int_M \varphi\,dV_g.\]
\end{enumerate}
\end{theorem}

The next characterization is formulated in terms of the resolvent.

\begin{theorem}\label{thm:fractional-conservative-equivalence}
Let \((M,g)\) be a complete Riemannian manifold, and let \(0<s<1\). Then the following assertions are equivalent.
\begin{enumerate}
\item[(i)] The subordinate semigroup \(\{T_t^{(s)}\}_{t\ge0}\) is conservative, that is, $T_t^{(s)}\mathbf 1=\mathbf 1$ for all $t>0.$

\item[(ii)] For every \(\alpha>0\),
\(R_\alpha^{(s)}\mathbf 1=\alpha^{-1},\)
where
\(R_\alpha^{(s)}:=\bigl((-\Delta)^s+\alpha\bigr)^{-1}.\)

\item[(iii)] For every \(\alpha>0\), the function
\(v_\alpha:=\mathbf 1-\alpha R_\alpha^{(s)}\mathbf 1\)
vanishes identically on \(M\).
\end{enumerate}
\end{theorem}

By Lemma \ref{lem:frac-lap-smooth-compact-support} and Lemma \ref{lem:fraclap-test-L1}, we know that  for every \(\varphi\in C_c^\infty(M)\),
\((-\Delta)^s\varphi\in C^\infty(M)\cap L^1(M).\) We then prove our main nonlocal characterization.

\begin{theorem}\label{thm:conservative-characterization-zero-mean}
Let \((M,g)\) be a complete Riemannian manifold, and let \(0<s<1\). Then the following assertions are equivalent:

\medskip
\noindent\textnormal{(i)} The subordinate semigroup \(\{T_t^{(s)}\}_{t\ge0}\) is conservative, namely, $T_t^{(s)}\mathbf 1=\mathbf 1$ for all $t>0.$

\medskip
\noindent\textnormal{(ii)} For every \(\varphi\in C_c^\infty(M)\),
\[
\int_M (-\Delta)^s\varphi\,dV_g=0.
\]
\end{theorem}

A remarkable feature of Theorem~\ref{thm:conservative-characterization-zero-mean}
is that it is genuinely nonlocal and has no analogue for the local Laplacian.
Indeed, for \(0<s<1\), the operator \((-\Delta)^s\) produces a nontrivial tail at infinity even when applied to compactly supported test functions, so that the above integral
captures global information about the geometry and the long-range behavior of the process.

Moreover, Corollary~\ref{cor:int-fraclap-Rs} shows that
\[
\int_M (-\Delta)^s\varphi\,dV_g
=
\int_M \mathcal R_s(x)\,\varphi(x)\,dV_g(x),
\]
where \(\mathcal R_s\) is the intrinsic killing term associated with the loss of mass of the heat semigroup. This identity makes the meaning of the theorem particularly transparent: the zero-mean condition in \textnormal{(ii)} exactly detects whether this mass-defect term vanishes, and therefore whether the subordinate semigroup is conservative.

Next, we recall the notion of a core for the generator of a semigroup and prove a classical equivalence between stochastic completeness and the \(L^1\)-core property of the heat semigroup generator. This result will play a crucial role in the proof of our final characterization theorem.

\begin{definition}\cite[Section 6.1]{davies2007linear}
Let \(X\) be a Banach space, let \(\{T_t\}_{t\ge0}\) be a strongly continuous semigroup on \(X\), and let \(A\) denote its generator with domain \(\operatorname{Dom}(A)\subset X\). A linear subspace
\(D\subset \operatorname{Dom}(A)\)
is called a \emph{core} for \(A\) if \(D\) is dense in \(\operatorname{Dom}(A)\) with respect to the graph norm
\(\|f\|_{A}:=\|f\|_X+\|Af\|_X.\)
\end{definition}

For the heat semigroup on a complete Riemannian manifold, the \(L^2\) and \(L^1\) situations are markedly different. In \(L^2(M)\), the situation is classical:  the operator \(-\Delta\), initially defined on \(C_c^\infty(M)\), is essentially self-adjoint. Thus, \(C_c^\infty(M)\) is a core for the \(L^2\)-realization of \(\Delta\), see \cite[Theorem 4.3]{strichartz1983analysis}, independently of whether \((M,g)\) is stochastically complete.

In \(L^1(M)\), the behavior is completely different. The generator \(\Delta\) of the heat semigroup on \(L^1(M)\) need not have \(C_c^\infty(M)\) as a core automatically. In fact, this turns out to be equivalent to stochastic completeness of the manifold. Thus, unlike the \(L^2\)-theory, the \(L^1\)-core property detects whether heat mass escapes at infinity.

Let \(\Delta_1\) and \(\Delta\) denote the generators of the heat semigroup on
\(L^1(M)\) and \(L^2(M)\), respectively. Since the semigroup actions coincide on
\(L^1(M)\cap L^2(M)\), the two generators are compatible in the sense that
\[
\Delta_1 f=\Delta f
\qquad\text{for every }f\in \operatorname{Dom}(\Delta_1)\cap \operatorname{Dom}(\Delta).
\]
Accordingly, when no confusion arises, we shall simply write \(\Delta\) for these realizations.

\begin{theorem}\label{thm:L1-core-stochastic-completeness}
Let \((M,g)\) be a complete Riemannian manifold, and let \(\Delta\) denote the generator of the heat semigroup on \(L^1(M)\). Then
\(C_c^\infty(M)\ \text{is a core for }\Delta\)
if and only if \((M,g)\) is stochastically complete.
\end{theorem}

As an application, we prove that stochastic completeness is equivalent to the uniqueness of bounded distributional solutions to a fractional elliptic equation. This is markedly different from the classical local case. Indeed, for the Laplace--Beltrami operator, one considers the homogeneous equation, reflecting the fact that formally \(\Delta 1=0\). In the fractional setting, by contrast, the natural equation is
\((-\Delta)^s v+\alpha v=\mathcal R_s,\)
which is consistent with the formal identity \((-\Delta)^s 1=\mathcal R_s\). Another major difference is that, in the local case, uniqueness is often established by pointwise arguments, whereas in the nonlocal setting it is more natural to work with distributional solutions. We therefore begin by introducing the notion of bounded distributional solution for the fractional elliptic equation.

\begin{definition}
Let \(0<s<1\), \(\alpha>0\), and let \(f\in L^1_{\mathrm{loc}}(M)\). A function \(u\in L^\infty(M)\) is called a distributional solution of
\[
(-\Delta)^s u+\alpha u=f
\qquad\text{in }M,
\]
if
\[
\int_M u\,(-\Delta)^s\varphi\,dV_g
+\alpha\int_M u\varphi\,dV_g
=
\int_M f\,\varphi\,dV_g
\qquad\forall\,\varphi\in C_c^\infty(M).
\]
\end{definition}

The above identity is meaningful. Indeed, since \(u\in L^\infty(M)\) and \(\varphi\in C_c^\infty(M)\), we have \(u\varphi\in L^1(M)\); since \(f\in L^1_{\mathrm{loc}}(M)\), it follows that \(f\varphi\in L^1(M)\); and, by Lemma~\ref{lem:fraclap-test-L1}, \((-\Delta)^s\varphi\in L^1(M)\), so that \(u(-\Delta)^s\varphi\in L^1(M)\) as well. 

\begin{theorem}\label{prop:unique-trivial-solution-equivalence}
Let \((M,g)\) be a complete Riemannian manifold, let \(0<s<1\), and let \(\alpha>0\). Then the following statements are equivalent:
\begin{enumerate}
\item[(i)] \(M\) is stochastically complete.

\item[(ii)] The only bounded distributional solution \(v\in L^\infty(M)\) of
\[
(-\Delta)^s v+\alpha v=\mathcal R_s
\qquad\text{in }M
\]
is the trivial one, namely $v=0
\quad\text{a.e. on }M.$
\end{enumerate}
\end{theorem}

\begin{remark}
Theorem ~\ref{prop:unique-trivial-solution-equivalence} should be compared with
the classical characterization of stochastic completeness for the Laplace--Beltrami
operator: in the local case, stochastic completeness is equivalent to the fact that
the equation
\[
-\Delta v+\alpha v=0
\qquad\text{in }M
\]
admits no nontrivial bounded solution. In the fractional case, however, the correct
characterization used here concerns, instead of the homogeneous equation
\((-\Delta)^s v+\alpha v=0,\)
the compensated equation
\((-\Delta)^s v+\alpha v=\mathcal R_s.\)
The reason is that the zeroth-order term \(\mathcal R_s\) precisely records the loss
of mass of the heat semigroup, and therefore plays the role of the correction term
needed to formulate the appropriate nonlocal counterpart of stochastic completeness.
\end{remark}

Finally, we establish our main characterization in terms of the associated elliptic and parabolic equations. Since, on a complete Riemannian manifold, the fractional Laplacian generally does not admit a convenient pointwise definition, it is more natural to treat it through functional calculus or, equivalently, through the generator of the subordinate semigroup. For this reason, we shall work throughout in the distributional setting. We begin by introducing the notions of bounded distributional solutions to the corresponding parabolic problems.

\begin{definition}
Let \(u_0\in L^\infty(M)\) and \(f\in L^\infty_{\mathrm{loc}}(M\times[0,\infty))\).
A function
\[
u\in L^\infty(M\times(0,T))\qquad\text{for every }0<T<\infty
\]
is called a distributional solution of the Cauchy problem
\[
\partial_t u+(-\Delta)^s u=f
\qquad\text{in }M\times(0,\infty),
\]
with initial datum
\(u(\cdot,0)=u_0,\)
if for every test function
\(\Phi\in C_c^\infty(M\times[0,\infty))\)
one has
\[
\begin{aligned}
&-\int_0^\infty\!\!\int_M
u(x,t)\,\partial_t\Phi(x,t)\,dV_g(x)\,dt
+\int_0^\infty\!\!\int_M
u(x,t)\,(-\Delta)^s\Phi(\cdot,t)(x)\,dV_g(x)\,dt \\
&\qquad
=\int_M u_0(x)\Phi(x,0)\,dV_g(x)
+\int_0^\infty\!\!\int_M
f(x,t)\Phi(x,t)\,dV_g(x)\,dt.
\end{aligned}
\]
\end{definition}
The global spatial boundedness ensures that the nonlocal term
is well defined, since $(-\Delta)^s\Phi(\cdot,t)$ is integrable
but generally not compactly supported. We can now state our final characterization theorem, which shows that the conservativeness of the subordinate fractional heat semigroup is equivalent to the uniqueness of bounded distributional solutions to both the associated elliptic equation and the fractional heat equation.

\begin{theorem}\label{maintheorem}
Let $(M,g)$ be a complete Riemannian manifold, let $0<s<1$, and let $\alpha>0$. Then the following assertions are equivalent:

\medskip
\noindent\textnormal{(i)} The subordinate fractional heat semigroup $\{T_t^{(s)}\}_{t\ge0}$ is conservative, namely
\[
T_t^{(s)}\mathbf 1=\mathbf 1
\qquad\text{for every }t>0.
\]

\medskip
\noindent\textnormal{(ii)} The only bounded distributional solution of the elliptic equation
\[
(-\Delta)^s u+\alpha u=\mathbf 1
\qquad\text{in }M
\]
is the constant function $u= \alpha^{-1}$ almost everywhere on \(M\).

\medskip
\noindent\textnormal{(iii)} The only bounded distributional solution of the fractional heat equation
\[
\begin{cases}
\partial_t w+(-\Delta)^s w=0 & \text{in }M\times(0,\infty),\\
w(\cdot,0)=\mathbf 1 & \text{on }M,
\end{cases}
\]
is the constant function $w(x,t)=1$ almost everywhere on \(M\times(0,\infty)\).
\end{theorem}

Besides the nonlocal characterizations of stochastic completeness established above, we also obtain several complementary results concerning subordinate semigroups, fractional heat kernels, small-time asymptotics of jump probabilities, and fractional resolvents. We begin with the basic \(L^p\)-theory of the subordinate semigroup. The following proposition shows that the fractional heat semigroup extends consistently to all \(L^p\)-spaces and preserves the contraction property of the classical heat semigroup.

\begin{proposition}\label{prop:fractional-semigroup-Lp}
Let \((M,g)\) be a complete Riemannian manifold, and let \(0<s<1\). Then, for every \(1\le p\le \infty\), the family
\(\{T_t^{(s)}\}_{t\ge0}=\{e^{-t(-\Delta)^s}\}_{t\ge0}\)
extends to a contraction semigroup on \(L^p(M)\), namely,
\[
T_0^{(s)}=\mathrm{id},
\qquad
T_{t+\tau}^{(s)}=T_t^{(s)}T_\tau^{(s)}
\quad\text{for all }t,\tau\ge0,
\]
and
\[
\|T_t^{(s)}f\|_{L^p(M)}\le \|f\|_{L^p(M)}
\qquad\text{for all }f\in L^p(M),\ t\ge0.
\]
Moreover, if \(1\le p<\infty\), then \(\{T_t^{(s)}\}_{t\ge0}\) is strongly continuous on \(L^p(M)\), that is,
\[
\lim_{t\rightarrow0}\|T_t^{(s)}f-f\|_{L^p(M)}=0
\qquad\text{for every }f\in L^p(M).
\]
\end{proposition}

The next result shows that, despite its nonlocal nature, the subordinate semigroup still enjoys a strong smoothing effect inherited from the classical heat semigroup.

\begin{proposition}\label{lem:subordinate-semigroup-smoothing}
Let \((M,g)\) be a complete Riemannian manifold, let \(0<s<1\), and let $T_t^{(s)}f$ be the subordinate semigroup. Then for every \(t>0\) and every \(f\in L^2(M)\), one has
\(T_t^{(s)}f\in C^\infty(M).\)
Moreover, for every chart \(U\Subset M\), every compact set \(K\Subset U\), and every integer \(m\ge0\),
\(\|T_t^{(s)}f\|_{C^m(K)} \le C_{K,U,m,t}\,\|f\|_{L^2(M)},\)
where $C_{K,U,m,t}>0$ independent of $f.$
\end{proposition}

We also investigate the global behavior of the fractional heat kernel. The following proposition identifies its large-time exponential decay rate in terms of the bottom of the spectrum, exactly paralleling the classical heat kernel asymptotics.

\begin{proposition}\label{prop:fractional-heat-kernel-long-time}
Let \((M,g)\) be a complete Riemannian manifold, let \(0<s<1\), and let
\(p_t^{(s)}(x,y)\) be the heat kernel associated with \(e^{-t(-\Delta)^s}\). 
Then, for every \(x,y\in M\),
\[
\lim_{t\to\infty}\frac{\log p_t^{(s)}(x,y)}{t}
=
-\lambda_0(M)^s,
\]
where
\(\lambda_0(M):=\inf \sigma(-\Delta).\)
\end{proposition}

At small time scales, the nonlocal heat kernel exhibits a fundamentally different behavior from the local one, see Proposition \ref{prop:varadhan-heat-kernel}. The next result shows that away from the diagonal, the leading term is linear in \(t\) and is precisely governed by the jump kernel \(K_s(x,y)\).

\begin{proposition}\label{prop:fractional-heat-kernel-small-time-offdiag}
Let \((M,g)\) be a complete Riemannian manifold, let \(0<s<1\), and let
\(p_t^{(s)}(x,y)\) be the heat kernel associated with \(e^{-t(-\Delta)^s}\). Then, for every \(x\neq y\),
\[
\lim_{t\rightarrow0}\frac{p_t^{(s)}(x,y)}{t}
=
K_s(x,y),
\]
where \(K_s(x,y)\) is defined in \eqref{frackernel}. 
\end{proposition}

This asymptotic formula admits a natural probabilistic interpretation. In fact, it yields the short-time asymptotics of jump probabilities for the Markov process generated by the fractional Laplacian. The contrast with the classical heat kernel is a direct consequence of the nonlocal nature of the operator \((-\Delta)^s\). For the local heat semigroup \(e^{t\Delta}\), the associated diffusion has continuous sample paths. Therefore, reaching a fixed point \(y\neq x\) in a very short time \(t\) is a highly atypical event, and the corresponding heat kernel is exponentially small, with leading behavior of the form
\[
\exp\!\left(-\frac{d_g(x,y)^2}{4t}\right).
\]
By contrast, the semigroup \(e^{-t(-\Delta)^s}\) corresponds to a jump process. Thus, for \(x\neq y\), the dominant contribution as \(t\downarrow0\) comes from the possibility of a single jump from \(x\) to \(y\), which is of order \(t\), and the coefficient is precisely the jump kernel \(K_s(x,y)\).

We now briefly recall the probabilistic interpretation of the fractional heat semigroup. Let \(0<s<1\), and let $\{T_t^{(s)}\}_{t\ge0}$ be the subordinate semigroup generated by \(-(-\Delta)^s\). Since this is a sub-Markov semigroup, there exists an associated Markov process \((X_t^{(s)})_{t\ge0},\) 
whose transition probabilities are described by the fractional heat kernel \(p_t^{(s)}(x,y)\). We adjoin a cemetery state to account for killing. More precisely, for every measurable set \(A\subset M\),
\[
\mathbb P_x\,\bigl(X_t^{(s)}\in A\bigr)
=
\int_A p_t^{(s)}(x,y)\,dV_g(y),
\qquad t>0.
\]
In contrast to the Brownian motion associated with the classical heat semigroup, the process \((X_t^{(s)})_{t\ge0}\) is a jump process. The next result shows that, away from the starting point \(x\), the short-time transition probability is governed by the jump kernel \(K_s(x,y)\).

\begin{proposition}\label{prop:jump-probability}
Let \((M,g)\) be a complete Riemannian manifold, let \(0<s<1\), and let
\((X_t^{(s)})_{t\ge0}\) be the Markov process associated with the semigroup
\((e^{-t(-\Delta)^s})_{t\ge0}\). Fix \(x\in M\), and let \(A\subset M\) be a measurable set such that $\operatorname{dist}(x,A)>0$ and $V_g(A)<\infty$.
Then
\[
\lim_{t\rightarrow0}\frac{1}{t}\,\mathbb P_x\!\bigl(X_t^{(s)}\in A\bigr)
=
\int_A K_s(x,y)\,dV_g(y),
\]
where $K_s(x,y)$ is defined in \eqref{frackernel}.
\end{proposition}

The kernel \(K_s(x,y)\) plays a central role in the probabilistic potential theory of the
fractional operator \((-\Delta)^s\). From the probabilistic point of view, \(K_s(x,y)\,dV_g(y)\)
is the jump intensity of the Markov process associated with the semigroup
\((e^{-t(-\Delta)^s})_{t\ge0}\): it describes the first-order rate at which the process jumps
from \(x\) to \(dy\). Consequently, \(K_s\) determines the short-time exit behavior of the process
from a domain \(D\), and hence it is closely related to the exit distribution of the stopped
process.

We next turn to the elliptic theory associated with the resolvent $R_\alpha^{(s)}$. For this purpose, we first introduce the natural notion of weak solution in the fractional Sobolev space \(H^s(M)\).

In the variational formulations below, we work with real-valued functions.

\begin{definition}\label{def:fractional-weak-solution}
Let $0<s<1$, $\alpha>0$, and $f\in L^2(M)$. Set
\[
\mathcal E_s(u,v):=\langle (-\Delta)^{\frac{s}{2}} u,(-\Delta)^{\frac{s}{2}} v\rangle_{L^2(M)},
\qquad
u,v\in H^s(M).
\]

\begin{enumerate}
\item[(i)] A function $u\in H^s(M)$ is called a \emph{weak solution} of
\begin{equation}\label{eq:fractional-elliptic}
(-\Delta)^{s}u+\alpha u=f
\end{equation}
if
\[
\mathcal E_s(u,\varphi)+\alpha\int_M u\varphi\,dV_g
=
\int_M f\varphi\,dV_g
\qquad\forall\,\varphi\in H^s(M).
\]

\item[(ii)] A function $u\in H^s(M)$ is called a \emph{weak supersolution} of
\eqref{eq:fractional-elliptic} if
\[
\mathcal E_s(u,\varphi)+\alpha\int_M u\varphi\,dV_g
\ge
\int_M f\varphi\,dV_g
\qquad\forall\,\varphi\in H^s(M),\ \varphi\ge0.
\]
\end{enumerate}
\end{definition}

With this notion in hand, the fractional resolvent can be characterized variationally. The next proposition shows that it yields the unique weak solution of the elliptic equation, and moreover it enjoys a minimality property among all nonnegative weak supersolutions.

\begin{proposition}\label{prop:fractional-resolvent-minimality}
Let $(M,g)$ be a complete Riemannian manifold, let $0<s<1$, $\alpha>0$, and let
$f\in L^2(M)$. Then the following assertions hold.

\begin{enumerate}
\item[(i)] The function
\(u_\alpha:=R_\alpha^{(s)}f=((-\Delta)^s+\alpha)^{-1}f\)
belongs to $H^{2s}(M)$ and is the unique weak solution of
\begin{equation}\label{eq:fractional-elliptic-resolvent}
(-\Delta)^s u+\alpha u=f
\qquad\text{on }M,
\end{equation}
that is,
\begin{equation}\label{eq:fractional-weak-resolvent}
\mathcal E_s(u_\alpha,\varphi)+\alpha\int_M u_\alpha\varphi\,dV_g
=
\int_M f\varphi\,dV_g
\qquad
\forall\,\varphi\in H^s(M).
\end{equation}

\item[(ii)]  If \(u\in H^s(M)\) is a weak supersolution of \(((-\Delta)^s+\alpha)u=f\), namely
\[
\mathcal E_s(u,\varphi)+\alpha\int_M u\varphi\,dV_g
\ge
\int_M f\varphi\,dV_g
\qquad
\forall\,\varphi\in H^s(M),\ \varphi\ge 0,
\]
then $u\ge R_\alpha^{(s)}f$ a.e. on $M.$
In particular, if \(f\ge 0\), then \(R_\alpha^{(s)}f\) is the minimal nonnegative weak supersolution.
\end{enumerate}
\end{proposition}

The remainder of the paper is organized as follows. In Section~2, we recall the basic facts on complete Riemannian manifolds and heat kernels, review several equivalent definitions of the fractional Laplacian, and introduce the subordinate semigroup associated with \((-\Delta)^s\). In Section~3, we develop a number of analytic and probabilistic results for the fractional Laplacian, including contraction and regularity properties of the subordinate semigroup, asymptotic behavior of the fractional heat kernel, small-time asymptotics of jump probabilities, the associated resolvent theory and the generalized conservation property. Section~4 contains the main results of the paper. In particular, we establish several equivalent characterizations of stochastic completeness in the nonlocal setting, including the equivalence between classical and fractional conservativeness, the resolvent and zero-mean criteria, the \(L^1\)-core criterion, and the elliptic and parabolic characterizations.

	%====================================================

\section{Preliminaries}

In this section, we recall the basic material needed in the sequel, including the heat semigroup, stochastic completeness, the fractional Laplacian, and several properties of the associated subordinate semigroup.

\subsection{Complete Riemannian Manifolds and the Heat Kernel}\label{comple}
Throughout this paper, we assume that \((M,g)\) is an $n$-dimensional connected, complete manifold without boundary. For brevity, we shall simply refer to such a manifold as a complete manifold. By the classical theorem of Gaffney and Chernoff \cite{gaffney1954heat,chernoff1973essential}, the minimal Laplace--Beltrami operator
\(-\Delta:\ C_c^\infty(M)\subset L^2(M)\to L^2(M)\)
is essentially self-adjoint. Consequently, \(-\Delta\) admits a unique self-adjoint realization on \(L^2(M)\).

By contrast, if \(M\) has boundary, the operator $-\Delta\big|_{C_c^\infty(M)}$
is in general not essentially self-adjoint, since different self-adjoint extensions may arise from different boundary conditions. Therefore, in the presence of a boundary, uniqueness of the self-adjoint realization typically requires incorporating an appropriate boundary condition into the initial domain. For example, in the work of Bianchi, Güneysu, and Setti \cite{bianchi2025neumann}, essential self-adjointness of the Laplacian with Neumann boundary condition is established by taking the domain to be \(\widehat{C}_c^\infty(M)\).

More generally, many of the results in this paper can also be formulated on incomplete Riemannian manifolds. In that setting, however, the Laplace--Beltrami operator defined initially on \(C_c^\infty(M)\) need not admit a unique self-adjoint extension. One must therefore first choose a specific self-adjoint realization of the Laplacian, corresponding for instance to a prescribed boundary behavior or extension procedure, and then develop the theory relative to that choice.

For basic background on Riemannian manifolds and the Laplace--Beltrami operator, we refer to
\cite{hebey2000nonlinear,rosenberg1997laplacian,jost2005riemannian,petersen2006riemannian,grigoryan2009heat,li2012geometric,lee2018introduction}.

  Next, we recall some basic facts about the heat kernel associated with the heat semigroup $P_t:=e^{t\Delta}$, $t\ge0$.
Since \(-\Delta\) is a nonnegative self-adjoint operator on \(L^2(M)\), the family \(\{P_t\}_{t\ge0}\) forms a strongly continuous contraction semigroup on \(L^2(M)\). More generally, \(\{P_t\}_{t\ge0}\) extends to a strongly continuous contraction semigroup on \(L^p(M)\) for every \(1\le p<\infty\); see \cite[Section 7.4]{grigoryan2009heat}. We next record the basic properties of the heat kernel \(p(t,x,y)\); see \cite[Theorem 7.13]{grigoryan2009heat}.

\begin{proposition}\label{prop:heat-kernel-basic}
For every \(t,s>0\) and \(x,y\in M\), one has $p(t,x,y)=p(t,y,x)>0$ and $\int_Mp(t,x,y)\,dV_g(y)\le1$, with equality if \(M\) is stochastically complete, and
\[ p(t+s,x,y)=\int_M p(t,x,z)p(s,z,y)\,dV_g(z).\]
Moreover, for every \(f\in L^2(M)\), $P_t f(x)=\int_M p(t,x,y)f(y)\,dV_g(y),$
and, for each fixed \(y\in M\), the map \((t,x)\mapsto p(t,x,y)\) belongs to
\(C^\infty((0,\infty)\times M)\) and solves $\partial_tp(t,x,y)=\Delta_xp(t,x,y)$.
\end{proposition}

In particular, for every \(t>0\) and every \(x,y\in M\),
\[
p(t,x,\cdot)\in L^2(M),
\qquad
p(t,\cdot,y)\in L^2(M),
\]
and by symmetry and the semigroup identity,
\begin{equation*}
\|p(t,x,\cdot)\|_{L^2(M)}^2
=
\int_M p(t,x,z)^2\,dV_g(z)
=
p(2t,x,x).
\end{equation*}

We next record, for later use, the large-time asymptotic behavior and the logarithmic small-time asymptotic of the heat kernel, the latter being the classical Varadhan formula.

\begin{proposition}\cite[Section 10.8]{grigoryan2009heat}\label{boundp}
Let
\(\lambda_0(M):=\inf\sigma(-\Delta)\ge0.\)
Then for all \(x,y\in M\) and all \(t\ge s>0\),
\begin{equation*}
p(t,x,y)\le \sqrt{p(s,x,x)\,p(s,y,y)}\,e^{-\lambda_0(M)(t-s)}.
\end{equation*}
\end{proposition}

\begin{proposition}\cite[Theorem 10.24]{grigoryan2009heat}\label{prop:heat-kernel-long-time}
Let \((M,g)\) be a complete Riemannian manifold, and let
\(\lambda_0(M):=\inf \sigma(-\Delta).\)
Then, for every \(x,y\in M\),
\[
\lim_{t\to\infty}\frac{\log p(t,x,y)}{t}=-\lambda_0(M).
\]
\end{proposition}

\begin{proposition} \cite[Eq. 15.45]{grigoryan2009heat} \cite{varadhan1967behavior}
\label{prop:varadhan-heat-kernel}
Let \((M,g)\) be a Riemannian manifold.
Then, for every \(x,y\in M\),
\[
\lim_{t\rightarrow0} 4t\,\log p(t,x,y)=-d_g(x,y)^2.
\]
where \(d_g(x,y)\) denotes the geodesic distance between \(x\) and \(y\) on \(M\).
\end{proposition}

\subsection{Spectral and Pointwise Representations of the
	Fractional Laplacian}\label{Equiva}

Let $A=-\Delta$ be the nonnegative self-adjoint Laplacian on $L^2(M)$.
Since $(M,g)$ is complete, $-\Delta$ is essentially self-adjoint on
$C_c^\infty(M)$. By the spectral theorem, there exists a unique
projection-valued measure $E$ such that
$A=\int_{[0,\infty)}\lambda\,dE(\lambda)$.
For $f,g\in L^2(M)$ and a Borel set $B\subset[0,\infty)$, write
$E_{f,g}(B):=\langle E(B)f,g\rangle_{L^2(M)}$.
In particular, $E_{f,f}$ is a positive finite measure with
$E_{f,f}([0,\infty))=\|f\|_{L^2(M)}^2$.

For $0<s<1$, the spectral fractional Laplacian is defined by
$(-\Delta)^s=\int_{[0,\infty)}\lambda^s\,dE(\lambda)$, with domain
\[
H^{2s}(M):=\operatorname{Dom}((-\Delta)^s)
=\left\{f\in L^2(M):
\int_{[0,\infty)}\lambda^{2s}\,dE_{f,f}(\lambda)<\infty\right\}.
\]
For $f\in H^{2s}(M)$ and $g\in L^2(M)$,
\[
\langle(-\Delta)^sf,g\rangle_{L^2(M)}
=\int_{[0,\infty)}\lambda^s\,dE_{f,g}(\lambda),
\qquad
\|(-\Delta)^sf\|_{L^2(M)}^2
=\int_{[0,\infty)}\lambda^{2s}\,dE_{f,f}(\lambda).
\]
Equipped with the graph norm
$\|f\|_{H^{2s}(M)}^2
=\|f\|_{L^2(M)}^2+\|(-\Delta)^sf\|_{L^2(M)}^2$,
the space $H^{2s}(M)$ is a Hilbert space.
Moreover, $C_c^\infty(M)$ is dense in $H^{2s}(M)$.
Indeed, spectral truncation shows that $\operatorname{Dom}(A)$ is
dense in $\operatorname{Dom}(A^s)$ in the graph norm.
Since $C_c^\infty(M)$ is a core for $A$ and
$\|A^s h\|_2^2\le\|h\|_2^2+\|Ah\|_2^2$ for
$h\in\operatorname{Dom}(A)$, the assertion follows.

By the spectral theorem, the Bochner representation is
\begin{equation}\label{eq:bochner-fractional}
	(-\Delta)^sf
	=c_s\lim_{\varepsilon\downarrow0,\,R\uparrow\infty}
	\int_\varepsilon^R(f-P_tf)\,\frac{dt}{t^{1+s}}
	\quad\text{in }L^2(M),
\end{equation}
for every $f\in H^{2s}(M)$, where
$P_t=e^{t\Delta}$ and $c_s:=s/\Gamma(1-s)$.
Thus, the integral is understood as a strong improper Bochner
integral. We next establish the corresponding pointwise formula.

Set $\Theta(t,x):=P_t\mathbf1(x)$ and recall
\[
K_s(x,y)=c_s\int_0^\infty p(t,x,y)\,\frac{dt}{t^{1+s}},
\quad x\ne y,
\qquad
\mathcal R_s(x)=c_s\int_0^\infty
\bigl(1-\Theta(t,x)\bigr)\frac{dt}{t^{1+s}}.
\]
Propositions~\ref{prop:varadhan-heat-kernel} and~\ref{boundp}
imply that $K_s(x,y)$ is finite for $x\ne y$.
For $f\in C_c^\infty(M)$, the heat-semigroup identity
$f-P_tf=-\int_0^tP_\tau\Delta f\,d\tau$, first in $L^2(M)$
and then pointwise by continuity, gives
\[
|f(x)-P_tf(x)|
\le\min\{t\|\Delta f\|_\infty,\,2\|f\|_\infty\}.
\]
Consequently,
\[
(-\Delta)^sf(x)
=c_s\int_0^\infty\bigl(f(x)-P_tf(x)\bigr)\frac{dt}{t^{1+s}}
\]
converges absolutely and uniformly in $x$ at both endpoints.
Its value is the continuous representative of the spectral
fractional Laplacian, by \eqref{eq:bochner-fractional}.

For a compact set $D\subset M$, choose $\chi\in C_c^\infty(M)$
with $0\le\chi\le1$ and $\chi=1$ on a neighborhood of $D$.
Then, for $x\in D$,
\[
0\le1-\Theta(t,x)\le1-P_t\chi(x)
\le\min\{1,t\|\Delta\chi\|_\infty\}.
\]
In particular, $\mathcal R_s\in L^\infty_{\mathrm{loc}}(M)$.
Using
\[
f(x)-P_tf(x)
=\int_M(f(x)-f(y))p(t,x,y)\,dV_g(y)
+f(x)(1-\Theta(t,x)),
\]
we obtain
\begin{equation}\label{eq:time-jump-killing-decomposition}
	\begin{aligned}
		(-\Delta)^sf(x)
		={}&c_s\int_0^\infty
		\left[\int_M(f(x)-f(y))p(t,x,y)\,dV_g(y)\right]
		\frac{dt}{t^{1+s}}+\mathcal R_s(x)f(x).
	\end{aligned}
\end{equation}
Here the spatial integral is taken first, and its subsequent
time integral is absolutely convergent by the preceding estimates.

We now identify the first term with a geodesic principal value.
Fix $x\in M$ and use normal coordinates $y=\exp_xv$ on a
sufficiently small ball $B_\rho(x)$, with $\rho<1$.
Write $r=|v|$ and $dV_g(y)=J_x(v)\,dv$, where
$J_x(v)=1+O(r^2)$, and set
\[
K_{s,\delta}(x,y)
=c_s\int_\delta^\infty p(t,x,y)\,\frac{dt}{t^{1+s}},
\qquad
q_{s,\delta}(v)
=c_s\int_\delta^\infty
(4\pi t)^{-n/2}e^{-r^2/(4t)}\,\frac{dt}{t^{1+s}}.
\]
The local heat-kernel estimates in
\cite[Lemma 2.18 and Proposition 2.19]{caselli2024fractional}
apply here after shrinking and rescaling the normal chart and
smoothly extending the local metric to a complete metric on
$\mathbb R^n$ that agrees with it on a smaller ball and is constant
outside a compact set. Shrinking the chart and, if necessary,
applying a common constant metric rescaling sufficiently close to
one ensure the local flatness hypotheses of Lemma~2.18 and the
coefficient bounds in Proposition~2.19 with $\ell=2$.
These results give, after scaling back, for $0<t<t_0$ and $0<r<\rho$,
\[
\left|p(t,x,\exp_xv)-(4\pi t)^{-n/2}e^{-r^2/(4t)}\right|
\le Ct^{(1-n)/2}e^{-cr^2/t}+Ce^{-c/t}.
\]
For $n=1$, the same estimate follows directly from the heat kernels
of the line and the circle.
For $t\ge t_0$, Proposition~\ref{boundp} bounds the heat kernel
uniformly on the coordinate ball. Since $q_{s,0}(v)=C_{n,s}r^{-n-2s}$ for a positive constant
$C_{n,s}$, integration of the preceding estimates yields,
uniformly in $\delta>0$,
\[
\begin{aligned}
	\left|J_x(v)K_{s,\delta}(x,\exp_xv)-q_{s,\delta}(v)\right|
	&\le C\int_0^{t_0}t^{-(n+1)/2-s}e^{-cr^2/t}\,dt
	+Cr^2q_{s,0}(v)+C\le Cr^{1-n-2s}.
\end{aligned}
\]
In particular,
\[
J_x(v)K_{s,\delta}(x,\exp_xv)\le Cr^{-n-2s},
\]
uniformly in $\delta>0$. Write $\ell_x(v)=df_x(v)$ and
$R_x(v)=f(x)-f(\exp_xv)+\ell_x(v)=O(r^2)$.
With $k_\delta(v):=J_x(v)K_{s,\delta}(x,\exp_xv)$, we have
\[
\bigl(f(x)-f(\exp_xv)\bigr)k_\delta(v)
=-\ell_x(v)q_{s,\delta}(v)
+R_x(v)k_\delta(v)
-\ell_x(v)\bigl(k_\delta(v)-q_{s,\delta}(v)\bigr).
\]
The first term integrates to zero over every centered ball or
annulus because $q_{s,\delta}$ is radial. The remaining terms are
bounded, uniformly in $\delta>0$, by $Cr^{2-n-2s}$, which is
locally integrable since $\int_0^\rho r^{1-2s}\,dr<\infty$.
Outside $B_\rho(x)$, the kernel $K_s(x,\cdot)$ is integrable:
choose $0\le\chi\le1$ supported in $B_\rho(x)$ and equal to one
near $x$, and use
\[
\int_{M\setminus B_\rho(x)}K_s(x,y)\,dV_g(y)
\le c_s\int_0^\infty(1-P_t\chi(x))\,\frac{dt}{t^{1+s}}<\infty.
\]
For each $\delta>0$, Fubini's theorem applies because
\[
c_s\int_\delta^\infty\int_M
|f(x)-f(y)|p(t,x,y)\,dV_g(y)\,\frac{dt}{t^{1+s}}
\le\frac{2c_s\|f\|_\infty}{s}\delta^{-s}.
\]
Dominated convergence for the compensated local integrand and
the integrable tail, followed by cancellation over centered annuli,
therefore gives
\begin{equation*}
	\begin{aligned}
		\operatorname{P.V.}\int_M(f(x)-f(y))K_s(x,y)\,dV_g(y)
		&=\lim_{\delta\downarrow0}
		\int_M(f(x)-f(y))K_{s,\delta}(x,y)\,dV_g(y)\\
		&=c_s\int_0^\infty
		\bigl(f(x)\Theta(t,x)-P_tf(x)\bigr)\frac{dt}{t^{1+s}},
	\end{aligned}
\end{equation*}
where the principal value is taken over $M\setminus B_\varepsilon(x)$
as $\varepsilon\downarrow0$.
Combining this identity with
\eqref{eq:time-jump-killing-decomposition}, we conclude that
\begin{equation*}
		(-\Delta)^sf(x)
		=\operatorname{P.V.}\int_M(f(x)-f(y))K_s(x,y)\,dV_g(y)
		+\mathcal R_s(x)f(x).
\end{equation*}

If $(M,g)$ is stochastically complete, then $\mathcal R_s\equiv0$,
and hence
\begin{equation*}
	(-\Delta)^sf(x)
	=\operatorname{P.V.}\int_M(f(x)-f(y))K_s(x,y)\,dV_g(y).
\end{equation*}
The spatial integral is absolutely convergent when $0<s<\frac12$.
For $\frac12\le s<1$, the principal value is necessary in general.

Finally, the decay of $K_s$ depends on the geometry of $M$.
For example, on the hyperbolic space $\mathbb H^n$, $n\ge2$,
one has \cite[Proposition 1.11]{chen2025logarithmic}
\[
K_s(x,y)\asymp
\begin{cases}
	r^{-n-2s}, & r\to0,\\
	r^{-1-s}e^{-(n-1)r}, & r\to\infty,
\end{cases}
\qquad r=d_g(x,y),
\]
where the comparison constants depend only on $n$ and $s$.

\subsection{Subordinate Semigroup Associated with the Fractional Laplacian}\label{subordinate}

For $0<s<1$, the subordinate semigroup in $L^2(M)$ generated by $-(-\Delta)^s$ is given by $T_t^{(s)}:=e^{-t(-\Delta)^s}$, $t\ge0$.
By Bochner's subordination formula, see \cite[Chapter 13]{schilling2012bernstein} and \cite[Section 5.4]{grigor2003heat}, there exists a family of probability densities $\eta_t^{(s)}(\tau)\ge0$, $\tau,t>0$, such that
\begin{equation}\label{integralres}
    e^{-t(-\Delta)^s}f
=
\int_0^\infty \eta_t^{(s)}(\tau)\,e^{\tau \Delta}f\,d\tau,
\qquad f\in L^2(M),
\end{equation}
where the integral is understood in the Bochner sense. Equivalently,
\begin{equation}\label{laplacets}
    \int_0^\infty e^{-\tau\lambda}\eta_t^{(s)}(\tau)\,d\tau
=
e^{-t\lambda^s},
\qquad \lambda\ge0.
\end{equation}
In particular, $\int_0^\infty\eta_t^{(s)}(\tau)\,d\tau=1$ for $t>0$.
Moreover, for every $0<s<1$, the density $\eta_t^{(s)}$ enjoys the scaling property
\begin{equation}\label{scaling pro}
    \eta_t^{(s)}(\tau)
=
\frac{1}{t^{1/s}}\,
\eta_1^{(s)}\!\left(\frac{\tau}{t^{1/s}}\right),
\qquad \tau>0,\ t>0,
\end{equation}
and satisfies the estimates
\begin{equation}\label{tailinte}
    \eta_t^{(s)}(\tau)\le C\,\frac{t}{\tau^{1+s}},
\qquad \forall\,\tau,t>0,
\end{equation}
where $C$ is a positive constant independent of $t,\tau$. Furthermore, as $\tau\to0^+$,
$\eta_1^{(s)}(\tau)$ decays exponentially fast:
\begin{equation}\label{eta1s}
    0\le \eta_1^{(s)}(\tau)\le C\tau^{-1-s}e^{-\tau^{-s}},\quad\tau>0.
\end{equation}
See \cite[proof of Theorem 3.1, equations (7)--(9), (14)]{bogdan2003harnack}, with $\alpha=2s$ there, and \cite[Section 5.4]{grigor2003heat}. In particular, equation (7) and the scaling formula \eqref{scaling pro} give
\begin{equation*}
\lim_{t\downarrow0}\frac{\eta_t^{(s)}(\tau)}t=c_s\tau^{-1-s},\qquad\tau>0.
\end{equation*}

The fractional heat kernel associated with \((-\Delta)^s\), equivalently the heat kernel of the subordinate semigroup \(e^{-t(-\Delta)^s}\), is given by
\begin{equation}\label{subordination formula}
    p_t^{(s)}(x,y)
:=
\int_0^\infty \eta_t^{(s)}(\tau)\,p(\tau,x,y)\,d\tau,
\end{equation}
see \cite[Section 5.4]{grigor2003heat}. The integral in \eqref{subordination formula} is finite for every $t>0$ and $x,y\in M$: the local estimate $p(\tau,z,z)=O_z(\tau^{-n/2})$ as $\tau\downarrow0$, Cauchy--Schwarz, \eqref{scaling pro}, and \eqref{eta1s} control its small-time part, since $\eta_t^{(s)}(\tau)\le Ct\tau^{-1-s}e^{-t\tau^{-s}}$; for $\tau\ge1$, the diagonal spectral bound $p(\tau,z,z)\le p(1,z,z)$ and the unit mass of $\eta_t^{(s)}$ control the remaining part. Tonelli's theorem and the convolution law of the subordination measures imply symmetry and the semigroup identity for $p_t^{(s)}$.
The subordinate semigroup $T_t^{(s)}$ also admits an integral kernel $p_t^{(s)}(x,y)$, see \cite{bogdan2003harnack,grigor2003heat,grigoryan2009heat,dong2023time}, given by
\[
T_t^{(s)}f(x)=\int_M p_t^{(s)}(x,y)f(y)\,dV_g(y),
\qquad f\in L^2(M).
\]
Indeed, by the symmetry of the fractional heat kernel and the semigroup property, we have
\[
\int_M p_t^{(s)}(x,y)^2\,dV_g(y)
=
\int_M p_t^{(s)}(x,y)\,p_t^{(s)}(y,x)\,dV_g(y)
=
p_{2t}^{(s)}(x,x).
\]

\section{Analytic Properties of the Fractional Laplacian}
In this section, we study several analytic and probabilistic aspects of the fractional Laplacian on complete Riemannian manifolds, including the \(L^p\)-theory and smoothing properties of the subordinate semigroup, asymptotic behavior of the fractional heat kernel, jump probabilities of the associated Markov process, and the variational characterization of the fractional resolvent.

\subsection{Contraction and Regularity of the Subordinate Semigroup}
In this subsection, we prove some basic \(L^p\)-contractivity and smoothing properties of the subordinate semigroup \(\{T_t^{(s)}\}_{t\ge0}\), which will be used repeatedly in the sequel.

The subordinate semigroup inherits both the \(L^p\)-contractivity of the heat semigroup and, for \(1\le p<\infty\), the corresponding strong continuity.

\begin{proof}[\textbf{Proof of Proposition \ref{prop:fractional-semigroup-Lp}.}]
The semigroup property on \(L^2(M)\) follows from the spectral definition:
\[
T_{t+\tau}^{(s)}
=
e^{-(t+\tau)(-\Delta)^s}
=
e^{-t(-\Delta)^s}e^{-\tau(-\Delta)^s}
=
T_t^{(s)}T_\tau^{(s)}.
\]
For \(1\le p<\infty\), the contractive estimate established below extends this identity from the dense subspace \(L^2(M)\cap L^p(M)\) to all of \(L^p(M)\). For \(p=\infty\), we use the kernel extension, equivalently the weak-star subordination formula. The convolution identity \(\eta_{t+\tau}^{(s)}=\eta_t^{(s)}*\eta_\tau^{(s)}\), which follows from \eqref{laplacets}, and Fubini's theorem give
\[
T_t^{(s)}T_\tau^{(s)}f
=\int_0^\infty\!\int_0^\infty
\eta_t^{(s)}(a)\eta_\tau^{(s)}(b)P_{a+b}f\,da\,db
=T_{t+\tau}^{(s)}f.
\]
Here Fubini is justified by \(|P_{a+b}f(x)|\le\|f\|_{L^\infty(M)}\).

We now prove the \(L^p\)-contractivity. For \(p=\infty\), by the subordination formula \eqref{integralres} and the \(L^\infty\)-contractivity of the heat semigroup,
\[
\|T_t^{(s)}f\|_{L^\infty(M)}
=
\left\|
\int_0^\infty \eta_t^{(s)}(\tau)\,P_\tau f\,d\tau
\right\|_{L^\infty(M)}
\le
\int_0^\infty \eta_t^{(s)}(\tau)\,\|P_\tau f\|_{L^\infty(M)}\,d\tau
\le
\|f\|_{L^\infty(M)}.
\]

Next, let \(f\in L^1(M)\). Using the subordination formula, Fubini's theorem, and the \(L^1\)-contractivity of the heat semigroup, we obtain
\[
\|T_t^{(s)}f\|_{L^1(M)}
\le
\int_0^\infty \eta_t^{(s)}(\tau)\,\|P_\tau f\|_{L^1(M)}\,d\tau
\le
\int_0^\infty \eta_t^{(s)}(\tau)\,\|f\|_{L^1(M)}\,d\tau
=
\|f\|_{L^1(M)}.
\]
The \(L^p\)-contractivity for \(1<p<\infty\) then follows by interpolation between \(L^1(M)\) and \(L^\infty(M)\). Therefore,
\[
\|T_t^{(s)}f\|_{L^p(M)}\le \|f\|_{L^p(M)}
\qquad\forall\,f\in L^p(M),\ \forall\,t\ge0,
\]
for every \(1\le p\le\infty\). Hence \(\{T_t^{(s)}\}_{t\ge0}\) is a contraction semigroup on \(L^p(M)\).

Let \(1\le p<\infty\) and \(f\in L^p(M)\). We prove strong continuity. By the subordination formula \eqref{integralres},
\[
T_t^{(s)}f-f
=
\int_0^\infty \eta_t^{(s)}(\tau)\,\bigl(P_\tau f-f\bigr)\,d\tau.
\]
Hence, by Minkowski's integral inequality,
\[
\|T_t^{(s)}f-f\|_{L^p(M)}
\le
\int_0^\infty \eta_t^{(s)}(\tau)\,\|P_\tau f-f\|_{L^p(M)}\,d\tau.
\]

Fix \(\varepsilon>0\). Since \(\{P_\tau\}_{\tau\ge0}\) is strongly continuous on \(L^p(M)\), see \cite[Theorem 7.19 and Exercise 7.37]{grigoryan2009heat}, there exists \(\delta>0\) such that
\[
\|P_\tau f-f\|_{L^p(M)}<\varepsilon
\qquad\text{for all }0<\tau<\delta.
\]
On the other hand, the \(L^p\)-contractivity of the heat semigroup gives
\(\|P_\tau f-f\|_{L^p(M)}\le2\|f\|_{L^p(M)}\).
Splitting the integral at \(\delta\), we obtain
\[
\|T_t^{(s)}f-f\|_{L^p(M)}
\le
\varepsilon
+
2\|f\|_{L^p(M)}
\int_\delta^\infty \eta_t^{(s)}(\tau)\,d\tau.
\]
By \eqref{tailinte}, we obtain for every fixed \(\delta>0\),
\[
\int_\delta^\infty \eta_t^{(s)}(\tau)\,d\tau\to0
\qquad\text{as }t\downarrow0.
\]
Hence
\[
\limsup_{t\rightarrow0}\|T_t^{(s)}f-f\|_{L^p(M)}\le \varepsilon.
\]
Since \(\varepsilon>0\) is arbitrary, it follows that
\[
\lim_{t\rightarrow0}\|T_t^{(s)}f-f\|_{L^p(M)}=0,
\]
we complete the proof.
\end{proof}

\begin{remark}
For \(1\le p<\infty\), let
\(-(-\Delta)^s_p\)
denote the infinitesimal generator of \(\{T_t^{(s)}\}_{t\ge0}\) on \(L^p(M)\). Namely,
\[
D\bigl((-\Delta)^s_p\bigr)
=
\left\{
f\in L^p(M):
\lim_{t\downarrow0}\frac{T_t^{(s)}f-f}{t}
\ \text{exists in }L^p(M)
\right\},
\]
and
\[
(-\Delta)^s_p f
=
-\lim_{t\downarrow0}\frac{T_t^{(s)}f-f}{t}.
\]
 Moreover, the \(L^p\)-generator is compatible with the spectral operator on \(L^2(M)\): if
\(f\in D\bigl((-\Delta)^s_p\bigr)\cap D\bigl((-\Delta)^s\bigr),\)
then
\((-\Delta)^s_p f = (-\Delta)^s f.\)
When no confusion is likely to arise, we shall simply write \((-\Delta)^s\) in place of \((-\Delta)^s_p\).
\end{remark}

Next, we show that the subordinate semigroup inherits the smoothing effect of the heat semigroup and maps \(L^2(M)\) into \(C^\infty(M)\) for every positive time.

\begin{proof}[\textbf{Proof of Proposition \ref{lem:subordinate-semigroup-smoothing}.}]
Fix \(t>0\), a chart \(U\Subset M\), a compact set \(K\Subset U\), and an integer \(m\ge0\).
By \cite[Theorem 7.6]{grigoryan2009heat}, there exist constants \(C>0\) and \(\sigma>\frac m2+\frac n4\) such that
\[
\|P_\tau f\|_{C^m(K)}
\le
C(1+\tau^{-\sigma})\|f\|_{L^2(M)}
\qquad\forall\,\tau>0.
\]
Hence
\[
\int_0^\infty \eta_t^{(s)}(\tau)\,\|P_\tau f\|_{C^m(K)}\,d\tau
\le
C\|f\|_{L^2(M)}
\int_0^\infty \eta_t^{(s)}(\tau)(1+\tau^{-\sigma})\,d\tau.
\]
It remains to show that
\[
\int_0^\infty \eta_t^{(s)}(\tau)(1+\tau^{-\sigma})\,d\tau<\infty.
\]
By the scaling property of the subordination kernel, \(\eta_t^{(s)}\) satisfies the same type of estimate as \(\eta_1^{(s)}\). In particular, there exist constants \(c,C_t>0\) such that
\[
0\le \eta_t^{(s)}(\tau)\le C_t\,\tau^{-1-s}e^{-c\tau^{-s}}
\qquad\text{for }0<\tau\le1.
\]
Therefore,
\[
\int_0^1 \eta_t^{(s)}(\tau)(1+\tau^{-\sigma})\,d\tau<\infty,
\]
since \(e^{-c\tau^{-s}}\) decays faster than any power as \(\tau\downarrow0\).
On the other hand, since \(\eta_t^{(s)}\) is a probability density,
\[
\int_1^\infty \eta_t^{(s)}(\tau)(1+\tau^{-\sigma})\,d\tau
\le
2\int_1^\infty \eta_t^{(s)}(\tau)\,d\tau
\le 2.
\]
Thus
\[
\int_0^\infty \eta_t^{(s)}(\tau)(1+\tau^{-\sigma})\,d\tau<\infty.
\]
It follows that the Bochner integral defining \(T_t^{(s)}f\) converges in \(C^m(K)\), and
\(\|T_t^{(s)}f\|_{C^m(K)} \le C_{K,U,m,t}\,\|f\|_{L^2(M)}.\)
Since \(m\) is arbitrary, we conclude that \(T_t^{(s)}f\in C^\infty(M)\).
\end{proof}

\begin{remark}\label{rem:continuity-vs-smoothness-subordinate}
Let \(0<s<1\), \(t>0\), and \(f\in L^\infty(M)\).

First, the continuity of \(T_t^{(s)}f\) follows directly from the subordination formula
\[
T_t^{(s)}f(x)=\int_0^\infty \eta_t^{(s)}(\tau)\,P_\tau f(x)\,d\tau.
\]
Indeed, for every \(\tau>0\), the function \(P_\tau f\) is continuous on \(M\), and by the \(L^\infty\)-contractivity of the heat semigroup,
\(\|P_\tau f\|_{L^\infty(M)}\le \|f\|_{L^\infty(M)}.\)
Hence, if \(x_j\to x\) in \(M\), then for every \(\tau>0\),
\(P_\tau f(x_j)\to P_\tau f(x),\)
while
\(|P_\tau f(x_j)|\le \|f\|_{L^\infty(M)}.\)
Since \(\eta_t^{(s)}(\tau)\,d\tau\) is a probability measure on \((0,\infty)\), the dominated convergence theorem yields
\(T_t^{(s)}f(x_j)\to T_t^{(s)}f(x).\)
Therefore,
\[
f\in L^\infty(M)\quad\Longrightarrow\quad T_t^{(s)}f\in C(M)
\qquad\text{for every }t>0.
\]

On the other hand, the smoothness of \(T_t^{(s)}f\) requires more care. Formally, one would like to write
\[
\nabla_x^\alpha T_t^{(s)}f(x)
=
\int_0^\infty \eta_t^{(s)}(\tau)\,\nabla_x^\alpha P_\tau f(x)\,d\tau,
\]
but this interchange of differentiation and integration is not justified solely by the fact that
\(P_\tau f\in C^\infty(M)\) for each \(\tau>0\). One also needs suitable local derivative estimates for the heat semigroup, namely, for every compact set \(K\subset M\),
\[
\int_0^\infty
\eta_t^{(s)}(\tau)\,
\sup_{x\in K}\bigl|\nabla^\alpha P_\tau f(x)\bigr|
\,d\tau
<\infty.
\]
Thus, to prove that \(T_t^{(s)}f\in C^\infty(M)\), one needs quantitative control of the derivatives
of \(P_\tau f\) as \(\tau\downarrow0\), and not merely the fact that \(P_\tau f\) is smooth for each fixed \(\tau>0\).
\end{remark}

\subsection{Asymptotic Behavior of the Fractional Heat Kernel}

By subordination formula, the long-time exponential decay of the fractional heat kernel
is determined by the bottom of the spectrum of \((-\Delta)^s\), namely
\(\lambda_0(M)^s\).

\begin{proof}[\textbf{Proof of Proposition \ref{prop:fractional-heat-kernel-long-time}.}]
Fix \(x,y\in M\). 
We split the proof into the upper and lower bounds.

\medskip
\noindent\textbf{(i) Upper bound.}
By the finiteness of the kernel in \eqref{subordination formula}, its symmetry, and the semigroup identity,
\[
\|p_1^{(s)}(z,\cdot)\|_{L^2(M)}^2=p_2^{(s)}(z,z)<\infty,
\qquad z\in M.
\]
For \(t>2\), the semigroup identity and the spectral theorem therefore give
\[
\begin{aligned}
p_t^{(s)}(x,y)
&=\bigl\langle p_1^{(s)}(x,\cdot),
T_{t-2}^{(s)}p_1^{(s)}(y,\cdot)\bigr\rangle_{L^2(M)}\\
&\le\sqrt{p_2^{(s)}(x,x)p_2^{(s)}(y,y)}\,
e^{-(t-2)\lambda_0(M)^s}.
\end{aligned}
\]
This estimate also applies when \(\lambda_0(M)=0\), and yields
\[
\limsup_{t\to\infty}\frac1t\log p_t^{(s)}(x,y)
\le
-\lambda_0(M)^s.
\]

\medskip
\noindent\textbf{(ii) Lower bound.}
Again let \(\varepsilon>0\). By Proposition \ref{prop:heat-kernel-long-time}, there exists
\(\tau_\varepsilon>0\) such that
\[
p(\tau,x,y)\ge e^{-(\lambda_0(M)+\varepsilon)\tau},
\qquad \forall\,\tau\ge \tau_\varepsilon.
\]
Hence
\[
p_t^{(s)}(x,y)
\ge
\int_{\tau_\varepsilon}^\infty \eta_t^{(s)}(\tau)p(\tau,x,y)\,d\tau
\ge
\int_{\tau_\varepsilon}^\infty \eta_t^{(s)}(\tau)e^{-(\lambda_0(M)+\varepsilon)\tau}\,d\tau.
\]
Therefore
\[
p_t^{(s)}(x,y)
\ge
\int_0^\infty \eta_t^{(s)}(\tau)e^{-(\lambda_0(M)+\varepsilon)\tau}\,d\tau
-
\int_0^{\tau_\varepsilon}\eta_t^{(s)}(\tau)e^{-(\lambda_0(M)+\varepsilon)\tau}\,d\tau.
\]
Using again the Laplace transform identity \eqref{laplacets}, we have
\[
p_t^{(s)}(x,y)
\ge
e^{-t(\lambda_0(M)+\varepsilon)^s}
-
\int_0^{\tau_\varepsilon}\eta_t^{(s)}(\tau)\,d\tau.
\]
To estimate the last term, fix \(q>\lambda_0(M)+\varepsilon\).
By \eqref{laplacets},
\[
\int_0^{\tau_\varepsilon}\eta_t^{(s)}(\tau)\,d\tau
\le e^{q\tau_\varepsilon}
\int_0^\infty e^{-q\tau}\eta_t^{(s)}(\tau)\,d\tau
=e^{q\tau_\varepsilon-tq^s}.
\]
Since \(q^s>(\lambda_0(M)+\varepsilon)^s\), this is at most
\(\tfrac12e^{-t(\lambda_0(M)+\varepsilon)^s}\) for all sufficiently large \(t\). Hence
\(p_t^{(s)}(x,y)\ge \frac12 e^{-t(\lambda_0(M)+\varepsilon)^s}.\)
Therefore
\[
\liminf_{t\to\infty}\frac1t\log p_t^{(s)}(x,y)
\ge
-(\lambda_0(M)+\varepsilon)^s.
\]
Letting \(\varepsilon\downarrow0\), we get
\[
\liminf_{t\to\infty}\frac1t\log p_t^{(s)}(x,y)
\ge
-\lambda_0(M)^s.
\]

Combining the upper and lower bounds, we conclude that
\[
\lim_{t\to\infty}\frac{\log p_t^{(s)}(x,y)}{t}
=
-\lambda_0(M)^s.
\]
This completes the proof.
\end{proof}

We next establish the off-diagonal small-time asymptotics of the fractional heat kernel, which will serve as the key analytic ingredient in the proof of the jump probability estimate.

\begin{proof}[\textbf{Proof of Proposition \ref{prop:fractional-heat-kernel-small-time-offdiag}.}]
Fix \(x,y\in M\) with \(x\neq y\). By the subordination formula,
\[
p_t^{(s)}(x,y)=\int_0^\infty \eta_t^{(s)}(\tau)\,p(\tau,x,y)\,d\tau,
\qquad t>0.
\]
Hence
\[
\frac{p_t^{(s)}(x,y)}{t}
=
\int_0^\infty \frac{\eta_t^{(s)}(\tau)}{t}\,p(\tau,x,y)\,d\tau.
\]

We claim that, for every fixed \(\tau>0\),
\[
\lim_{t\rightarrow0}\frac{\eta_t^{(s)}(\tau)}{t}
=
\frac{s}{\Gamma(1-s)}\,\tau^{-1-s}.
\]
Indeed, by the scaling property of the stable subordinator density \eqref{scaling pro},
\[
\eta_t^{(s)}(\tau)
=
t^{-1/s}\eta_1^{(s)}\!\left(\tau t^{-1/s}\right).
\]
Let \(r=\tau t^{-1/s}\). Then \(r\to\infty\) as \(t\rightarrow0\). By applying \cite[Theorem 37.1]{doetsch2012introduction} to the Laplace transform $F(z)=e^{-z^s},$
whose expansion at \(z=0\) is
\(e^{-z^s}=1-z^s+\frac12 z^{2s}+O(z^{3s}),\)
we obtain
\[
\eta_1^{(s)}(r)\sim 
 \frac{s}{\Gamma(1-s)}\,r^{-1-s}
\qquad\text{as } r\to\infty.
\]
See also \cite[proof of Theorem 3.1, equations (7)--(9)]{bogdan2003harnack} for this stable-density asymptotic.
Therefore,
\[
\eta_t^{(s)}(\tau)
=
t^{-1/s}\eta_1^{(s)}(r)
\sim
t^{-1/s}\frac{s}{\Gamma(1-s)}\,r^{-1-s}
=\frac{ts}{\Gamma(1-s)}\,\tau^{-1-s},
\]
which proves the claim.

On the other hand, by the pointwise estimate for \(\eta_t^{(s)}\), there exists a constant \(C=C(s)>0\)
such that
\[
0\le \eta_t^{(s)}(\tau)\le C\,\frac{t}{\tau^{1+s}},
\qquad \forall\,\tau,t>0.
\]
Hence
\(0\le \frac{\eta_t^{(s)}(\tau)}{t}\,p(\tau,x,y) \le C\,\frac{p(\tau,x,y)}{\tau^{1+s}}.\)
Since $\int_0^\infty \frac{p(\tau,x,y)}{\tau^{1+s}}\,d\tau<\infty,$
the dominated convergence theorem applies and yields
\[
\lim_{t\rightarrow0}\frac{p_t^{(s)}(x,y)}{t}
=
\int_0^\infty
\lim_{t\rightarrow0}\frac{\eta_t^{(s)}(\tau)}{t}\,p(\tau,x,y)\,d\tau
=
\frac{s}{\Gamma(1-s)}\int_0^\infty \frac{p(\tau,x,y)}{\tau^{1+s}}\,d\tau=K_s(x,y).
\]
This proves the desired limit.
\end{proof}

\subsection{Small-Time Asymptotics of Jump Probabilities}

The next proof shows that the off-diagonal kernel asymptotics obtained above yields the precise first-order behavior of the jump probabilities of the subordinate process.

\begin{proof}[\textbf{Proof of Proposition \ref{prop:jump-probability}.}]
Since \(X_t^{(s)}\) is the Markov process, with a cemetery state if killing occurs, associated with the semigroup
\((e^{-t(-\Delta)^s})_{t\ge0}\), its transition probabilities admit the density
\(p_t^{(s)}(x,y)\) with respect to \(dV_g(y)\). Therefore,
\[
\mathbb P_x\!\bigl(X_t^{(s)}\in A\bigr)
=
\int_A p_t^{(s)}(x,y)\,dV_g(y).
\]
Hence
\[
\frac{1}{t}\,\mathbb P_x\!\bigl(X_t^{(s)}\in A\bigr)
=
\int_A \frac{p_t^{(s)}(x,y)}{t}\,dV_g(y).
\]

Now let \(y\in A\). Since \(\operatorname{dist}(x,A)>0\), we have \(y\neq x\). By
Proposition \ref{prop:fractional-heat-kernel-small-time-offdiag},
\[
\lim_{t\rightarrow0}\frac{p_t^{(s)}(x,y)}{t}=K_s(x,y),
\qquad y\in A.
\]
Thus it remains to justify the passage of the limit under the integral sign.

Using the subordination formula,
\[
p_t^{(s)}(x,y)=\int_0^\infty \eta_t^{(s)}(\tau)\,p(\tau,x,y)\,d\tau,
\]
and the bound \eqref{tailinte},
\[
\eta_t^{(s)}(\tau)\le C_s\,\frac{t}{\tau^{1+s}},
\qquad \tau,t>0,
\]
we obtain
\[
0\le \frac{p_t^{(s)}(x,y)}{t}
\le
C_s\int_0^\infty p(\tau,x,y)\,\frac{d\tau}{\tau^{1+s}}
=
C_s\,\frac{\Gamma(1-s)}{s}\,K_s(x,y),
\qquad y\in A.
\]
Therefore,
\[
0\le \frac{p_t^{(s)}(x,y)}{t}\le C\,K_s(x,y),
\qquad y\in A,
\]
for some constant \(C=C(s)>0\), uniformly in \(t>0\).

It remains to show that \(K_s(x,\cdot)\in L^1(A,dV_g)\).
Choose \(0<r<\operatorname{dist}(x,A)\) and
\(\chi\in C_c^\infty(B_r(x))\) with \(0\le\chi\le1\) and \(\chi(x)=1\).
Since \(\chi=0\) on \(A\), sub-Markovianity and the heat-semigroup identity for \(\chi\) give
\[
\begin{aligned}
\int_A p(\tau,x,y)\,dV_g(y)
&\le1-P_\tau\chi(x)
=-\int_0^\tau P_u\Delta\chi(x)\,du\le\tau\|\Delta\chi\|_{L^\infty(M)}.
\end{aligned}
\]
By Tonelli's theorem, it follows that
\[
\int_A\int_0^1 p(\tau,x,y)\,\frac{d\tau}{\tau^{1+s}}\,dV_g(y)
\le\|\Delta\chi\|_{L^\infty(M)}\int_0^1\tau^{-s}\,d\tau
=\frac{\|\Delta\chi\|_{L^\infty(M)}}{1-s}<\infty.
\]

For \(\tau\ge1\), by the sub-Markov property, $\int_M p(\tau,x,y)\,dV_g(y)\le 1,$
and therefore
\[
\int_1^\infty \int_A p(\tau,x,y)\,dV_g(y)\,\frac{d\tau}{\tau^{1+s}}
\le
\int_1^\infty \int_M p(\tau,x,y)\,dV_g(y)\,\frac{d\tau}{\tau^{1+s}}
\le
\int_1^\infty \tau^{-1-s}\,d\tau
<\infty.
\]
Combining the estimates on \((0,1]\) and \([1,\infty)\), we obtain $\int_A K_s(x,y)\,dV_g(y)<\infty.$
The dominated convergence theorem now yields
\[
\lim_{t\rightarrow0}\frac{1}{t}\,\mathbb P_x\!\bigl(X_t^{(s)}\in A\bigr)
=
\lim_{t\rightarrow0}\int_A \frac{p_t^{(s)}(x,y)}{t}\,dV_g(y)
=
\int_A \lim_{t\rightarrow0}\frac{p_t^{(s)}(x,y)}{t}\,dV_g(y)
=
\int_A K_s(x,y)\,dV_g(y).
\]
This completes the proof.
\end{proof}

\subsection{Resolvent of the Fractional Laplacian}\label{resolvent}

For any $\alpha>0$, the resolvent of $(-\Delta)^s$ is a bounded operator on $L^2(M)$ defined by
\(R_\alpha^{(s)}:=\bigl((-\Delta)^s+\alpha\bigr)^{-1}.\)
Moreover, by the spectral theorem, the resolvent gains $2s$ derivatives, in the sense that
\(R_\alpha^{(s)}:L^2(M)\to \operatorname{Dom}((-\Delta)^s)=H^{2s}(M)\)
is bounded. We next record the following useful \(L^2\)-pairing identity for the fractional resolvent.

\begin{lemma}\label{lem:fractional-resolvent-laplace}
Let $(M,g)$ be a complete Riemannian manifold, let $0<s<1$, and let $\alpha>0$. 
Then, for every $f,g\in L^2(M)$,
\begin{equation}\label{eq:fractional-resolvent-laplace-pairing}
\bigl\langle R_\alpha^{(s)}f,g\bigr\rangle_{L^2(M)}
=
\int_0^\infty e^{-\alpha t}\bigl\langle T_t^{(s)}f,g\bigr\rangle_{L^{2}(M)}\,dt.
\end{equation}
\end{lemma}

\begin{proof}
By the spectral theorem,
\[
T_t^{(s)}f=\int_{[0,\infty)} e^{-t\lambda^s}\,dE(\lambda) f,
\qquad
R_\alpha^{(s)}f=\int_{[0,\infty)} \frac{1}{\lambda^s+\alpha}\,dE(\lambda) f.
\]
Hence, for every $f,g\in L^2(M)$,
\[
\bigl\langle T_t^{(s)}f,g\bigr\rangle_{L^{2}(M)}
=
\int_{[0,\infty)} e^{-t\lambda^s}\,dE_{f,g}(\lambda).
\]
Therefore,
by Fubini's theorem, whose absolute integrability follows from
\(|E_{f,g}|([0,\infty))\le\|f\|_{L^2(M)}\|g\|_{L^2(M)}\),
\[
\int_0^\infty e^{-\alpha t}\bigl\langle T_t^{(s)}f,g\bigr\rangle_{L^{2}(M)}\,dt
=
\int_{[0,\infty)}
\left(
\int_0^\infty e^{-t(\alpha+\lambda^s)}\,dt
\right)\,dE_{f,g}(\lambda)
=
\int_{[0,\infty)} \frac{1}{\alpha+\lambda^s}\,dE_{f,g}(\lambda),
\]
which is exactly $\bigl\langle R_\alpha^{(s)}f,g\bigr\rangle_{L^2(M)}.$
This proves \eqref{eq:fractional-resolvent-laplace-pairing}.
\end{proof}

By Proposition \ref{prop:fractional-semigroup-Lp}, for every \(1\le p\le\infty\), the family
\(\{T_t^{(s)}\}_{t\ge0}\) is contractive on \(L^p(M)\). Hence, for every \(\alpha>0\) and every
\(f\in L^p(M)\), the map $t\mapsto e^{-\alpha t}T_t^{(s)}f$
is Bochner integrable as an \(L^p(M)\)-valued function on \((0,\infty)\). For \(p<\infty\), strong measurability follows from strong continuity; for \(p=\infty\), norm continuity at every positive time follows from Lemma~\ref{lem:Linf-time-continuity-subordinate} below, whose proof uses only subordination. Therefore
\[\mathcal R_{\alpha,p}^{(s)}f
:=
\int_0^\infty e^{-\alpha t}T_t^{(s)}f\,dt\]
is well defined in \(L^p(M)\). Moreover,
\[
\|\mathcal R_{\alpha,p}^{(s)}f\|_{L^p(M)}
\le
\int_0^\infty e^{-\alpha t}\|T_t^{(s)}f\|_{L^p(M)}\,dt
\le
\int_0^\infty e^{-\alpha t}\,dt\,\|f\|_{L^p(M)}
=
\alpha^{-1}\|f\|_{L^p(M)}.
\]
Thus \(\mathcal R_{\alpha,p}^{(s)}:L^p(M)\to L^p(M)\) is a bounded linear operator satisfying $\|\mathcal R_{\alpha,p}^{(s)}\|_{\mathcal L(L^p(M))}
\le \alpha^{-1}.$

\smallskip

In the Hilbert space case \(p=2\), applying a bounded linear functional to the Bochner integral and using Lemma \ref{lem:fractional-resolvent-laplace} gives
\[
\left\langle \mathcal R_{\alpha,2}^{(s)}f,g\right\rangle_{L^2(M)}
=
\left\langle R_\alpha^{(s)}f,g\right\rangle_{L^2(M)}
\qquad\forall\, g\in L^2(M),
\]
and hence
\[R_\alpha^{(s)}f
=
\mathcal R_{\alpha,2}^{(s)}f
=
\int_0^\infty e^{-\alpha t}T_t^{(s)}f\,dt
\qquad\text{in }L^2(M).\]

Accordingly, for general \(1\le p\le\infty\), we define the \(L^p\)-extension of the resolvent by
\[
R_\alpha^{(s)}f
:=
\int_0^\infty e^{-\alpha t}T_t^{(s)}f\,dt,
\qquad f\in L^p(M),
\]
where the integral is understood in the Bochner sense. With this definition, $R_\alpha^{(s)}:L^p(M)\to L^p(M)$
is a bounded linear operator and
\(\|R_\alpha^{(s)}f\|_{L^p(M)} \le \alpha^{-1}\|f\|_{L^p(M)}.\)

It is not difficult to verify that the fractional resolvent \(R_\alpha^{(s)}\) is positivity preserving and sub-Markovian. Namely, if \(f\in L^2(M)\) satisfies \(f\ge 0\) almost everywhere on \(M\), then
\[
R_\alpha^{(s)}f\ge 0
\qquad \text{a.e. on }M,
\]
while if \(f\le 1\) almost everywhere on \(M\), then $R_\alpha^{(s)}f\le \alpha^{-1}$ a.e. on $M.$

We finally turn to the variational characterization of the fractional resolvent, which provides both the existence and uniqueness of weak solutions and the minimality of the resolvent among nonnegative weak supersolutions.

\begin{proof}[\textbf{Proof of Proposition \ref{prop:fractional-resolvent-minimality}.}]
(i) Since $R_\alpha^{(s)}=((-\Delta)^s+\alpha)^{-1}$ is the resolvent of the nonnegative
self-adjoint operator $(-\Delta)^s$ on $L^2(M)$, we have
\(u_\alpha:=R_\alpha^{(s)}f\in \operatorname{Dom}((-\Delta)^s)=H^{2s}(M),\)
and
\[
((-\Delta)^s+\alpha)u_\alpha=f
\qquad\text{in }L^2(M).
\]
Therefore, for every $\varphi\in H^s(M)$,
\[
\langle (-\Delta)^s u_\alpha,\varphi\rangle_{L^2(M)}
+\alpha\langle u_\alpha,\varphi\rangle_{L^2(M)}
=
\langle f,\varphi\rangle_{L^2(M)}.
\]
Using the spectral theorem,
\[
\langle (-\Delta)^s u_\alpha,\varphi\rangle_{L^2(M)}
=
\langle (-\Delta)^{\frac s2}u_\alpha,(-\Delta)^{\frac s2}\varphi\rangle_{L^2(M)}
=
\mathcal E_s(u_\alpha,\varphi),
\]
hence \eqref{eq:fractional-weak-resolvent} follows. Uniqueness of weak solutions follows by
testing the homogeneous equation with the solution itself.

(ii) Let $w:=R_\alpha^{(s)}f.$
By part (i), $w$ is a weak solution of \eqref{eq:fractional-elliptic-resolvent}. Subtracting
the weak formulation for $w$ from the weak supersolution inequality for $u$, we obtain
\begin{equation}\label{eq:super-minus-sol}
\mathcal E_s(u-w,\varphi)+\alpha\int_M (u-w)\varphi\,dV_g\ge0\quad  \forall\,\varphi\in H^s(M),\ \varphi\ge0.
\end{equation}

Set $v:=u-w\in H^s(M),$
and define the positive and negative parts of \(v\) by
\[
v^+:=\max\{v,0\},
\qquad
v^-:=\max\{-v,0\}.
\]
Then
\[
v=v^+-v^-,
\qquad
|v|=v^++v^-,
\qquad
v^-=(w-u)^+.
\]

Since \(\mathcal E_s\) is a Dirichlet form on \(L^2(M)\), its domain \(H^s(M)\) is stable under truncations. Hence $v^+,\,v^-\in H^s(M);$
see \cite[Theorem 1.4.1]{fukushima2011dirichlet}. Moreover, $\mathcal E_s(v^+,v^-)\le 0.$
Indeed, by the Markov property of Dirichlet forms, $\mathcal E_s(|v|,|v|)\le \mathcal E_s(v,v).$
Using the identities
\[
v=v^+-v^-,
\qquad
|v|=v^++v^-,
\]
we compute
\(\mathcal E_s(|v|,|v|) = \mathcal E_s(v^+,v^+)+\mathcal E_s(v^-,v^-)+2\mathcal E_s(v^+,v^-),\)
and
\(\mathcal E_s(v,v) = \mathcal E_s(v^+,v^+)+\mathcal E_s(v^-,v^-)-2\mathcal E_s(v^+,v^-).\)
Therefore, $\mathcal E_s(v^+,v^-)\le 0.$ Choosing $\varphi=v^-\ge0$ in \eqref{eq:super-minus-sol}, we get
\[
\mathcal E_s(v,v^-)+\alpha\int_M vv^-\,dV_g\ge0.
\]
Now write $v=v^+-v^-$. Then
\(\mathcal E_s(v,v^-) = \mathcal E_s(v^+,v^-)-\mathcal E_s(v^-,v^-) \le -\mathcal E_s(v^-,v^-),\)
and
\[
\int_M vv^-\,dV_g
=
\int_M (v^+-v^-)v^-\,dV_g
=
-\int_M |v^-|^2\,dV_g.
\]
Hence
\[
0
\le
\mathcal E_s(v,v^-)+\alpha\int_M vv^-\,dV_g
\le
-\mathcal E_s(v^-,v^-)-\alpha\|v^-\|_{L^2(M)}^2.
\]
Therefore,
\[
\mathcal E_s(v^-,v^-)=0
\qquad\text{and}\qquad
\|v^-\|_{L^2(M)}=0,
\]
so $v^-=0$ almost everywhere, that is, $u-w\ge0$ a.e. on $M.$
Since $w=R_\alpha^{(s)}f$, this proves $u\ge R_\alpha^{(s)}f$ a.e. on $M.$
The last statement is immediate.
\end{proof}

\subsection{Generalized Conservation via the Intrinsic Killing Term}

Recall that, for every \(\varphi\in C_c^\infty(M)\),
\begin{equation*}
(-\Delta)^s \varphi
=
(-\Delta)_{\mathrm{hk}}^s \varphi+\mathcal R_s\,\varphi.
\end{equation*}
This decomposition shows that the spectral fractional Laplacian consists of a nonlocal
jump part, described by the heat-kernel operator \((-\Delta)_{\mathrm{hk}}^s\), together with a zeroth-order term \(\mathcal R_s\). The latter reflects the loss of mass of the underlying heat semigroup and may therefore be interpreted as an intrinsic killing term.

To prove these identities, we first establish an auxiliary lemma. Although \(\{T_t^{(s)}\}_{t\ge0}\) is in general not strongly continuous on \(L^\infty(M)\) at \(t=0\), it is still locally continuous in time away from the origin.

\begin{lemma}\label{lem:Linf-time-continuity-subordinate}
Let \(0<s<1\) and \(f\in L^\infty(M)\). Then, for every \(t_0>0\),
\[
\|T_t^{(s)}f-T_{t_0}^{(s)}f\|_{L^\infty(M)}\to0
\qquad\text{as }t\to t_0.
\]
In particular, for every \(x\in M\), the map
\((0,\infty)\ni t\longmapsto T_t^{(s)}f(x)\)
is continuous.
\end{lemma}

\begin{proof}
By the subordination formula,
\[
T_t^{(s)}f(x)=\int_0^\infty \eta_t^{(s)}(\tau)\,P_\tau f(x)\,d\tau.
\]
Since \(\{P_\tau\}_{\tau>0}\) is sub-Markovian,
\[
|P_\tau f(x)|\le \|f\|_{L^\infty(M)}
\qquad\forall \tau>0,\ x\in M.
\]
Hence
\[
\|T_t^{(s)}f-T_{t_0}^{(s)}f\|_{L^\infty(M)}
\le
\|f\|_{L^\infty(M)}
\|\eta_t^{(s)}-\eta_{t_0}^{(s)}\|_{L^1(0,\infty)}.
\]
It remains to show that
\[
\|\eta_t^{(s)}-\eta_{t_0}^{(s)}\|_{L^1(0,\infty)}\to0
\qquad\text{as }t\to t_0.
\]
Using the scaling property
\[
\eta_t^{(s)}(\tau)=t^{-1/s}\eta_1^{(s)}\!\left(\tau t^{-1/s}\right),
\]
this follows from the continuity in \(L^1(0,\infty)\) of the dilation map
\[
a\longmapsto a\,g(a\cdot),\qquad a>0,
\]
for \(g\in L^1(0,\infty)\), which in turn is a direct consequence of the classical continuity of translations in \(L^1(\mathbb R)\). Indeed, if one sets
\[
h(y):=e^y g(e^y),\qquad y\in\mathbb R,
\]
then \(h\in L^1(\mathbb R)\) and $\|h\|_{L^1(\mathbb R)}=\|g\|_{L^1(0,\infty)}.$
Therefore, after the change of variables \(x=e^y\),
\(\|D_a g-D_b g\|_{L^1(0,\infty)} = \|h(\cdot+\log a)-h(\cdot+\log b)\|_{L^1(\mathbb R)}.\)
Hence the continuity of \(a\mapsto D_a g\) in \(L^1(0,\infty)\) follows from the standard continuity of translations in \(L^1(\mathbb R)\). Therefore
\(\|T_t^{(s)}f-T_{t_0}^{(s)}f\|_{L^\infty(M)}\to0,\)
and the same estimate holds pointwise for the kernel representatives, proving the last assertion.
\end{proof}

\begin{remark}\label{rem:HtNs-continuity}
	The function $\mathcal R_s$ is nonnegative, locally bounded, and continuous on $M$.
	Indeed, for any compact set $K\subset M$, choose $\chi\in C_c^\infty(M)$
	with $0\le\chi\le1$ and $\chi=1$ on a neighborhood of $K$.
	By the heat-semigroup identity,
	\[
	0\le1-P_\tau\mathbf1(x)
	\le1-P_\tau\chi(x)
	=-\int_0^\tau P_u\Delta\chi(x)\,du
	\le\tau\|\Delta\chi\|_\infty,\qquad x\in K.
	\]
	Together with $1-P_\tau\mathbf1\le1$, this gives an integrable local
	majorant in \eqref{remainder}. Since $P_\tau\mathbf1$ is continuous
	for every $\tau>0$, dominated convergence proves the continuity of
	$\mathcal R_s$.
	
	For any nonnegative measurable function $h$, we extend the subordinate
	semigroup and resolvent by
	\[
	T_t^{(s)}h=\lim_{m\to\infty}T_t^{(s)}(h\wedge m),
	\qquad
	R_\alpha^{(s)}h=\int_0^\infty e^{-\alpha t}T_t^{(s)}h\,dt,
	\]
	where the values may initially be infinite.
\end{remark}

\begin{proof}[\textbf{Proof of Theorem \ref{thm:generalized identity}.}]
	Fix $x\in M$. Let $X$ be the minimal diffusion generated by $\Delta$,
	started at $x$, with lifetime $\zeta$, and let $S$ be an independent
	$s$-stable subordinator. Its Laplace transform is
	$\mathbb E[e^{-\lambda S_t}]=e^{-t\lambda^s}$, and its L\'evy measure satisfies
	\[
	\nu_s(dr)=\frac{s}{\Gamma(1-s)}r^{-1-s}\,dr,
	\qquad
	\overline\nu_s(a):=\nu_s((a,\infty))
	=\frac{a^{-s}}{\Gamma(1-s)},\quad a>0,
	\]
	where $\overline\nu_s(\infty)=0$. Since $P_r\mathbf1(x)=\mathbb P_x(\zeta>r)$, Tonelli's theorem gives
	\[
	\mathcal R_s(x)
	=\mathbb E_x\overline\nu_s(\zeta),
	\qquad
	T_t^{(s)}\mathbf1(x)=\mathbb P_x(S_t<\zeta).
	\]
	The subordinator $S$ has zero drift, infinite L\'evy measure, and
	$S_t\to\infty$ almost surely. By the range theorem
	\cite[Theorem 1.9(i)]{bertoin2000subordinators},
	$\mathbb P(a\in\overline{\{S_u:u\ge0\}})=0$ for every $a>0$.
	Independence of $S$ and $\zeta$ therefore implies that, on
	$\{\zeta<\infty\}$, the level $\zeta$ is almost surely crossed by
	a unique jump of $S$. Consequently,
	\[
	1-T_t^{(s)}\mathbf1(x)
	=\mathbb E_x\sum_{0<u\le t}
	\mathbf1_{\{S_{u-}<\zeta<S_u\}}.
	\]
	
	Enlarge the filtration of $S$ initially by $\zeta$. By independence,
	its jump measure still has compensator $du\,\nu_s(dr)$.
	Applying the compensation formula to the nonnegative predictable kernel
	\(H(u,r):=\mathbf1_{\{S_{u-}<\zeta<\infty\}}\mathbf1_{\{r>\zeta-S_{u-}\}}\) gives
	\begin{equation}\label{eq:jump-compensation-conservation}
		\begin{aligned}
			1-T_t^{(s)}\mathbf1(x)
			&=\mathbb E_x\int_0^t
			\mathbf1_{\{S_{u-}<\zeta<\infty\}}
			\overline\nu_s(\zeta-S_{u-})\,du\\
			&=\int_0^t\mathbb E_x\left[
			\mathbf1_{\{S_u<\zeta\}}
			\overline\nu_s(\zeta-S_u)\right]du.
		\end{aligned}
	\end{equation}
	Here the integrand is defined to be zero outside the indicated event.
	The second equality follows since $S_{u-}=S_u$ almost surely
	for each deterministic $u$. No prior integrability assumption is needed,
	as the compensation formula applies to nonnegative predictable functions.
	
	For every $q\ge0$, the Markov property of the killed diffusion and
	the lifetime representation of $\mathcal R_s$ imply
	\[
	\mathbb E_x\left[
	\mathbf1_{\{q<\zeta\}}\overline\nu_s(\zeta-q)
	\mid\mathcal F_q^X\right]
	=\mathbf1_{\{q<\zeta\}}\mathcal R_s(X_q).
	\]
	This identity remains valid for the possibly unbounded function
	$\overline\nu_s$ by truncation and monotone convergence.
	Conditioning on $S_u$ and using independence, we obtain
	\[
	\begin{aligned}
		\mathbb E_x\left[
		\mathbf1_{\{S_u<\zeta\}}\overline\nu_s(\zeta-S_u)\right]
		&=\int_0^\infty\eta_u^{(s)}(q)
		\mathbb E_x\left[
		\mathbf1_{\{q<\zeta\}}\mathcal R_s(X_q)\right]dq\\
		&=\int_0^\infty\eta_u^{(s)}(q)P_q\mathcal R_s(x)\,dq
		=T_u^{(s)}\mathcal R_s(x).
	\end{aligned}
	\]
	All these identities involve nonnegative integrands, so Tonelli's theorem
	applies without assuming that $\mathcal R_s$ is globally bounded.
	Substituting into \eqref{eq:jump-compensation-conservation} yields
	\[
		T_t^{(s)}\mathbf1(x)
		+\int_0^tT_u^{(s)}\mathcal R_s(x)\,du=1,
		\qquad t>0,\ x\in M.
	\]
	In particular, the time integral is finite and bounded by $1$.
	
	Finally, for $\alpha>0$, multiplying this identity by $e^{-\alpha t}$
	and integrating over $(0,\infty)$, we obtain, by Tonelli's theorem,
	\[
	\begin{aligned}
		\frac1\alpha
		&=R_\alpha^{(s)}\mathbf1(x)
		+\int_0^\infty e^{-\alpha t}
		\int_0^tT_u^{(s)}\mathcal R_s(x)\,du\,dt=R_\alpha^{(s)}\mathbf1(x)
		+\frac1\alpha R_\alpha^{(s)}\mathcal R_s(x).
	\end{aligned}
	\]
	Therefore,
	\[
		\alpha R_\alpha^{(s)}\mathbf1(x)
		+R_\alpha^{(s)}\mathcal R_s(x)=1,
		\qquad \alpha>0,\ x\in M.
	\]
	In particular, $0\le R_\alpha^{(s)}\mathcal R_s(x)\le1$.
	This completes the proof.
\end{proof}

\section{Nonlocal Characterization of Stochastic Completeness}
\subsection{Equivalence Between Classical and Fractional Conservativeness}
In this section we study the natural notion of stochastic completeness associated with the spectral fractional Laplacian \((-\Delta)^s\), namely the conservativeness of the subordinate semigroup
\[
T_t^{(s)}:=e^{-t(-\Delta)^s},\qquad t\ge0.
\]
By Remark \ref{rem:continuity-vs-smoothness-subordinate}, the function \(T_t^{(s)}\mathbf 1\) is bounded and continuous on \(M\) for every \(t>0\).
We say that the subordinate semigroup \(\{T_t^{(s)}\}_{t\ge0}\) is \emph{conservative} if
\[
T_t^{(s)}\mathbf 1=\mathbf 1
\qquad\text{for all }t>0.
\]

We begin with the first characterization theorem, which identifies the conservativeness of the subordinate semigroup with both the classical stochastic completeness of the heat semigroup and the preservation of total \(L^1\)-mass.

\begin{proof}[\textbf{Proof of Theorem \ref{prop:conservativeness-subordination}.}]
We first prove that for every \(t>0\) and every \(\varphi\in L^1(M),\,f\in L^{\infty}(M)\),
\begin{equation*}
\int_M \varphi \,T_t^{(s)}f\,dV_g
=
\int_M f\,T_t^{(s)}\varphi\,dV_g.
\end{equation*}
Indeed, the sub-Markov property gives
\[
\int_M\int_M p_t^{(s)}(x,y)|\varphi(x)f(y)|\,dV_g(y)\,dV_g(x)
\le\|\varphi\|_{L^1(M)}\|f\|_{L^\infty(M)}.
\]
Thus Fubini's theorem and the symmetry of \(p_t^{(s)}\) yield the asserted identity.

We next show that (i) implies (ii). Assume that $P_t\mathbf 1=\mathbf 1$ for all $t>0.$
Then, by the subordination formula and the fact that \(\eta_t^{(s)}\) is a probability density on \((0,\infty)\),
\[
T_t^{(s)}\mathbf 1
=
\int_0^\infty \eta_t^{(s)}(\tau)\,P_\tau\mathbf 1\,d\tau
=
\int_0^\infty \eta_t^{(s)}(\tau)\,d\tau
=
\mathbf 1
\]
for every \(t>0\). Thus (ii) holds.

We now prove that (ii) implies (i). Assume that $T_t^{(s)}\mathbf 1=\mathbf 1$ for all $t>0.$
Fix \(x\in M\) and set $u(\tau):=P_\tau\mathbf 1(x),\,\tau>0.$
Since \(\{P_t\}_{t\ge0}\) is sub-Markovian, we have
\[
0\le u(\tau)\le 1
\qquad\text{for all }\tau>0.
\]
For each \(t>0\),
\[
\mathbf 1
=
T_t^{(s)}\mathbf 1
=
\int_0^\infty \eta_t^{(s)}(\tau)\,u(\tau)\,d\tau
=
\int_0^\infty \eta_t^{(s)}(\tau)\,d\tau.
\]
Hence
\[
\int_0^\infty \eta_t^{(s)}(\tau)\,\bigl(1-u(\tau)\bigr)\,d\tau=0.
\]
Since \(\eta_t^{(s)}(\tau)>0\) for almost every \(\tau>0\), it follows that
\[
u(\tau)=1
\qquad\text{for almost every }\tau>0.
\]
Finally, for every \(x\in M\), the map \(\tau\mapsto P_\tau\mathbf 1(x)\) is continuous; see \cite[Theorem 7.16]{grigoryan2009heat}. Therefore, $P_\tau\mathbf 1=\mathbf 1$ for all $\tau>0,$
that is, (i) holds. Finally, we prove the equivalence between (ii) and (iii). If (ii) holds, then by the identity proved in the first part,
\[
\int_M T_t^{(s)}\varphi\,dV_g
=
\int_M T_t^{(s)}\mathbf 1\,\varphi\,dV_g
=
\int_M \varphi\,dV_g
\]
for every \(t>0\) and every \(\varphi\in L^1(M)\). Thus (iii) follows.

Conversely, if (iii) holds, then for every \(t>0\) and every \(\varphi\in L^1(M)\),
\[
\int_M T_t^{(s)}\mathbf 1\,\varphi\,dV_g
=
\int_M T_t^{(s)}\varphi\,dV_g
=
\int_M \varphi\,dV_g.
\]
Hence
\[
\int_M \bigl(T_t^{(s)}\mathbf 1-\mathbf 1\bigr)\varphi\,dV_g=0
\qquad \forall\,\varphi\in L^1(M).
\]
Therefore, $T_t^{(s)}\mathbf 1=\mathbf 1$ almost everywhere on $M.$
Since \(T_t^{(s)}\mathbf 1\) is continuous, (ii) holds. This completes the proof.
\end{proof}

\subsection{Resolvent Characterization}

Similar to Remark \ref{rem:continuity-vs-smoothness-subordinate}, we know that $R_\alpha^{(s)}\mathbf 1$ is bounded and continuous on $M.$ Next, we prove the second characterization theorem.

\begin{proof}[\textbf{Proof of Theorem \ref{thm:fractional-conservative-equivalence}.}]
We prove
\[
\textnormal{(i)}\Longrightarrow \textnormal{(ii)}
\Longrightarrow \textnormal{(iii)}
\Longrightarrow \textnormal{(i)}.
\]

\medskip
\noindent\textbf{\textnormal{(i)} \(\Longrightarrow\) \textnormal{(ii)}.}
Assume that \(\{T_t^{(s)}\}_{t\ge0}\) is conservative. Since \(\mathbf 1\in L^\infty(M)\), by the
\(L^\infty\)-resolvent representation,
\[
R_\alpha^{(s)}\mathbf 1
=
\int_0^\infty e^{-\alpha t}T_t^{(s)}\mathbf 1\,dt
=
\int_0^\infty e^{-\alpha t}\,dt
=
\alpha^{-1}.
\]

\medskip
\noindent\textbf{\textnormal{(ii)} \(\Longrightarrow\) \textnormal{(iii)}.}
This is immediate from the definition $v_\alpha=\mathbf 1-\alpha R_\alpha^{(s)}\mathbf 1.$

\medskip
\noindent\textbf{\textnormal{(iii)} \(\Longrightarrow\) \textnormal{(i)}.}
Assume that for every \(\alpha>0\),
\(v_\alpha=\mathbf 1-\alpha R_\alpha^{(s)}\mathbf 1\equiv0.\)
Then $\alpha R_\alpha^{(s)}\mathbf 1=\mathbf 1,$
that is,
\[
\int_0^\infty \alpha e^{-\alpha t}T_t^{(s)}\mathbf 1\,dt=\mathbf 1
\qquad\text{in }L^\infty(M).
\]

Since \(\{T_t^{(s)}\}_{t\ge0}\) is sub-Markovian, its kernel representatives satisfy
\(0\le T_t^{(s)}\mathbf 1(x)\le1\) for all \(t>0\) and \(x\in M\).
For every \(x\in M\), the resolvent identity gives
\[
\int_0^\infty \alpha e^{-\alpha t}\bigl(1-T_t^{(s)}\mathbf 1(x)\bigr)\,dt=0.
\]
Since the integrand is nonnegative and the weight \(\alpha e^{-\alpha t}\) is strictly positive for every \(t>0\), it follows that
\[
T_t^{(s)}\mathbf 1(x)=1
\qquad\text{for a.e. }t>0.
\]
By Lemma~\ref{lem:Linf-time-continuity-subordinate}, the map
\(t\mapsto T_t^{(s)}\mathbf1(x)\) is continuous on \((0,\infty)\), hence
\[
T_t^{(s)}\mathbf 1=\mathbf 1
\qquad\text{for all }t>0.
\]
Thus \(\{T_t^{(s)}\}_{t\ge0}\) is conservative.
The argument uses only one \(\alpha>0\); in particular,
\(R_\alpha^{(s)}\mathbf 1=\alpha^{-1}\) for one parameter already implies conservativeness.
\end{proof}

\subsection{Nonlocal Zero-Mean Characterization}

In this subsection, we first show that for every \(\varphi\in C_c^\infty(M)\),
\((-\Delta)^s\varphi\in C^\infty(M)\cap L^1(M).\)
This ensures that the global integral $\int_M (-\Delta)^s\varphi\,dV_g$
is well defined for compactly supported test functions. We then prove the main zero-mean characterization, namely, that the conservativeness of the subordinate semigroup is equivalent to the identity
\[
\int_M (-\Delta)^s\varphi\,dV_g=0
\qquad\forall\,\varphi\in C_c^\infty(M).
\]

\begin{lemma}\label{lem:frac-lap-smooth-compact-support}
Let \((M,g)\) be a complete Riemannian manifold, let \(0<s<1\), and let
\(\varphi\in C_c^\infty(M)\). Then $(-\Delta)^s\varphi\in C^\infty(M).$
\end{lemma}

\begin{proof}
By the Bochner formula,
\[
(-\Delta)^s\varphi
=
\frac{s}{\Gamma(1-s)}
\int_0^\infty \frac{\varphi-e^{t\Delta}\varphi}{t^{1+s}}\,dt,
\]
where the integral is understood pointwise since $\varphi\in C_c^\infty(M).$

Fix a compact set $K$ contained in a chart $U$, and let \(D\) be any differential operator on \(U\).
We shall prove that
\[
D(-\Delta)^s\varphi(x)
=
\frac{s}{\Gamma(1-s)}
\int_0^\infty \frac{D\varphi(x)-De^{t\Delta}\varphi(x)}{t^{1+s}}\,dt
\qquad\text{for }x\in U,
\]
and that the right-hand side defines a continuous function on \(K\), this yields \((-\Delta)^s\varphi\in C^\infty(M)\).

Write
\[
(-\Delta)^s\varphi
=
\frac{s}{\Gamma(1-s)}
\left(
\int_0^1 \frac{\varphi-e^{t\Delta}\varphi}{t^{1+s}}\,dt
+
\int_1^\infty \frac{\varphi-e^{t\Delta}\varphi}{t^{1+s}}\,dt
\right)
=: \frac{s}{\Gamma(1-s)}(I_1+I_2).
\]

We study the two terms separately. Put \(A=-\Delta\) on \(L^2(M)\).
Since \(\varphi\in C_c^\infty(M)\), it belongs to \(\operatorname{Dom}(A^m)\)
for every integer \(m\ge1\). Interior elliptic estimates and local Sobolev
embedding show that, for \(m\) sufficiently large depending on \(K\) and \(D\),
\[
\sup_{x\in K}|Dh(x)|\le C_{K,D}\sum_{j=0}^m\|A^jh\|_{L^2(M)},
\qquad h\in\operatorname{Dom}(A^m).
\]
The spectral calculus gives \(A^jP_t\varphi=P_tA^j\varphi\), so the right-hand
side is uniformly bounded for \(h=P_t\varphi\), \(t\ge0\). The same observation
applies to \(h=\Delta P_t\varphi\), using one additional power of \(A\).

\medskip
\noindent\textbf{Step 1: the large-time part \(I_2\).}
Since \(\varphi\in L^2(M)\), \cite[Theorem 7.6]{grigoryan2009heat} yields that for every \(t>0\), $e^{t\Delta}\varphi\in C^\infty(M),$
and moreover, for every compact set \(K\Subset U\),
\(\sup_{x\in K}|D(e^{t\Delta}\varphi)(x)| \le F_{K,D}(t)\|\varphi\|_{L^2(M)},\)
where \(F_{K,D}(t)\) is uniformly bounded on \([1,\infty)\).
Indeed, the preceding local estimate and
\(\sup_{\lambda\ge0}\lambda^je^{-t\lambda}\le\sup_{\lambda\ge0}\lambda^je^{-\lambda}<\infty\)
for \(t\ge1\) give this bound for every fixed \(j\). Thus
\(|D\varphi(x)-D(e^{t\Delta}\varphi)(x)|\)
is bounded in \((t,x)\in [1,\infty)\times K\). Since \(t^{-1-s}\in L^1(1,\infty)\), it follows that
for every compact \(K\Subset U\),
\(\frac{|D\varphi(x)-D(e^{t\Delta}\varphi)(x)|}{t^{1+s}}\)
admits an integrable dominating function on \((1,\infty)\), uniformly for \(x\in K\).

\medskip
\noindent\textbf{Step 2: the small-time part \(I_1\).}
Since \(\varphi\in C_c^\infty(M)\subset \operatorname{Dom}(-\Delta)\), the map $t\mapsto e^{t\Delta}\varphi$
is differentiable in \(L^2(M)\), and
\(\frac{d}{dt}e^{t\Delta}\varphi=\Delta e^{t\Delta}\varphi.\)
Therefore,
\[
e^{t\Delta}\varphi-\varphi
=
\int_0^t\Delta e^{\tau\Delta}\varphi\,d\tau
\qquad\text{in }L^2(M).
\]
This identity also holds in the graph norm of \(\operatorname{Dom}(A^m)\)
for every \(m\), because \(\varphi\in\operatorname{Dom}(A^{m+1})\).
The preceding local estimate therefore permits application of \(D\), giving
\[
D(e^{t\Delta}\varphi-\varphi)
=
\int_0^t D\Delta e^{\tau\Delta}\varphi\,d\tau
\]
pointwise on \(K\), with the time integral converging uniformly there.

Now we claim that there exist \(C_{K,D,\varphi}>0\) and \(t_0\in(0,1]\) such that
\[
\sup_{x\in K}|D\varphi(x)-D(e^{t\Delta}\varphi)(x)|
\le C_{K,D,\varphi}\, t
\qquad\forall\,0<t\le t_0.
\]
Indeed, strong continuity in \(L^2(M)\) gives
\(A^j(P_\tau\varphi-\varphi)=P_\tau A^j\varphi-A^j\varphi\to0\)
for every \(j\). The local estimate above consequently gives
\(P_\tau\varphi\to\varphi\) in \(C^\infty(K)\) as \(\tau\downarrow0\). Hence
\[
D\Delta e^{\tau\Delta}\varphi \to D\Delta\varphi
\qquad\text{uniformly on }K.
\]
Therefore there exist \(t_0\in(0,1]\) and \(C_{K,D,\varphi}>0\) such that
\(\sup_{0\le \tau\le t_0}\sup_{x\in K}|D\Delta e^{\tau\Delta}\varphi(x)| \le C_{K,D,\varphi}.\)
Consequently,
\[
\sup_{x\in K}|D\varphi(x)-D(e^{t\Delta}\varphi)(x)|
\le
\int_0^t \sup_{x\in K}|D\Delta e^{\tau\Delta}\varphi(x)|\,d\tau
\le C_{K,D,\varphi}\,t
\]
for all \(0<t\le t_0\). Enlarging the constant if necessary, the same estimate holds for all \(0<t\le1\).

It follows that for \(x\in K\) and \(0<t\le1\),
\(\frac{|D\varphi(x)-D(e^{t\Delta}\varphi)(x)|}{t^{1+s}} \le C_{K,D,\varphi}\, t^{-s}.\)
Since \(0<s<1\), the function \(t^{-s}\) is integrable on \((0,1)\). 

Combining the small-time and large-time parts, by the dominated convergence theorem, we conclude that for every differential operator \(D\) on \(U\),
\[
D(-\Delta)^s\varphi(x)
=
\frac{s}{\Gamma(1-s)}
\int_0^\infty \frac{D\varphi(x)-D(e^{t\Delta}\varphi)(x)}{t^{1+s}}\,dt
\]
exists and defines a continuous function on \(U\). Since the compact sets \(K\) and the differential operator \(D\) were arbitrary, this shows that $(-\Delta)^s\varphi\in C^\infty(M).$
Thus, we complete the proof.
\end{proof}

To prove the $(-\Delta)^s\varphi\in L^1(M),\,\varphi\in C_c^{\infty}(M)$, we need the following well known result. 

\begin{proposition}\cite[Theorem 3 and the following heat-kernel estimate]{grigor1994integral}\label{prop:heat-tail-integral}
Let \((M,g)\) be a complete Riemannian manifold, let \(A\subset M\) be a compact set, and let \(R>0\).
Define
\(A^R := \{ x \in M : d(x, A) < R \}\)
Then, for every \(t>0\),
\begin{equation*}
\int_A\int_{M\setminus A^R} p(t,x,y)\,dV_g(y)\,dV_g(x)
\le
\sqrt{\mathrm{Vol}_g(A)\mathrm{Vol}_g(A^R\setminus A)}
\max\!\left\{\frac{R}{\sqrt{2t}},\frac{\sqrt{2t}}{R}\right\}
\exp\!\left(-\frac{R^2}{4t}+\frac12\right).
\end{equation*}
\end{proposition}

\begin{lemma}\label{lem:fraclap-test-L1}
Let \((M,g)\) be a complete Riemannian manifold, let \(0<s<1\), and let \(\varphi\in C_c^\infty(M)\). Then $(-\Delta)^s\varphi\in L^1(M).$
\end{lemma}

\begin{proof}
Let $A:=\operatorname{supp}\varphi.$ Choose \(R>0\) and set $U:=A^R.$ Then \(A\Subset U\Subset M\). Since \(\varphi\in C_c^\infty(M)\), by Lemma~\ref{lem:frac-lap-smooth-compact-support} we know that $(-\Delta)^s\varphi\in C^\infty(M).$
In particular, $(-\Delta)^s\varphi\in L^1(U),$
because \(U\) has finite volume. Therefore, it remains to prove that $(-\Delta)^s\varphi\in L^1(M\setminus U).$

\smallskip

For every \(x\in M\setminus U\), since \(x\notin \operatorname{supp}\varphi\), we have \(\varphi(x)=0\). Hence, by the Bochner formula,
\[
(-\Delta)^s\varphi(x)
=-\frac{s}{\Gamma(1-s)}
\int_0^\infty \frac{e^{t\Delta}\varphi(x)}{t^{1+s}}\,dt.
\]
Integrating the absolute value over \(M\setminus U\) and using Tonelli's theorem, we obtain
\[
\int_{M\setminus U}|(-\Delta)^s\varphi(x)|\,dV_g(x)
\le
\frac{s}{\Gamma(1-s)}
\int_0^\infty \frac{1}{t^{1+s}}
\left(\int_{M\setminus U}|e^{t\Delta}\varphi(x)|\,dV_g(x)\right)\,dt.
\]
Now \(e^{t\Delta}\varphi(x)=\int_Ap(t,x,y)\varphi(y)\,dV_g(y)\).
By positivity, Fubini's theorem and the symmetry of the heat kernel,
\[
\begin{aligned}
\int_{M\setminus U}|e^{t\Delta}\varphi(x)|\,dV_g(x)
&\le\int_A|\varphi(y)|
\left(\int_{M\setminus U}p(t,y,x)\,dV_g(x)\right)dV_g(y)\\
&\le\|\varphi\|_{L^\infty(M)}
\int_A\int_{M\setminus A^R}p(t,x,y)\,dV_g(y)\,dV_g(x).
\end{aligned}
\]
Applying Proposition~\ref{prop:heat-tail-integral}, we infer that
\[
\int_{M\setminus U}|e^{t\Delta}\varphi(x)|\,dV_g(x)
\le
C_{\varphi,A,R}
\max\!\left\{\frac{R}{\sqrt{2t}},\frac{\sqrt{2t}}{R}\right\}
\exp\!\left(-\frac{R^2}{4t}+\frac12\right),
\]
where \(C_{\varphi,A,R}:=
\|\varphi\|_{L^\infty(M)}\sqrt{\mathrm{Vol}_g(A)\mathrm{Vol}_g(A^R\setminus A)}\).
We use this estimate for small times, and the \(L^1\)-contraction property for large times.
For \(0<t\le1\), we have
\[
\max\!\left\{\frac{R}{\sqrt{2t}},\frac{\sqrt{2t}}{R}\right\}
\le C_R\,t^{-1/2},
\]
so that
\[
\int_0^1\frac{1}{t^{1+s}}
\left(\int_{M\setminus U}|e^{t\Delta}\varphi(x)|\,dV_g(x)\right)dt
\le C_{\varphi,A,R}'\int_0^1t^{-3/2-s}e^{-R^2/(4t)}\,dt<\infty.
\]
Here the exponential factor dominates every negative power of \(t\) as \(t\downarrow0\).
For \(t\ge1\), we have
\[
\int_{M\setminus U}|e^{t\Delta}\varphi(x)|\,dV_g(x)
\le\int_A|\varphi(y)|\left(\int_Mp(t,x,y)\,dV_g(x)\right)dV_g(y)
\le\|\varphi\|_{L^1(M)}.
\]
Hence
\[
\int_1^\infty\frac{1}{t^{1+s}}
\left(\int_{M\setminus U}|e^{t\Delta}\varphi(x)|\,dV_g(x)\right)dt
\le\|\varphi\|_{L^1(M)}\int_1^\infty t^{-1-s}\,dt<\infty.
\]
Combining the small-time estimate with the large-time estimate, we conclude that
\[
\int_{M\setminus U}|(-\Delta)^s\varphi|\,dV_g<\infty.
\]
Since \((-\Delta)^s\varphi\in L^1(U)\), it follows that $(-\Delta)^s\varphi\in L^1(M).$

We also record the generator-domain fact needed below. The heat-semigroup identity
\(P_t\varphi-\varphi=\int_0^tP_r\Delta\varphi\,dr\) holds in \(L^2(M)\)
and in \(L^1(M)\); the two Bochner integrals agree by consistency of the semigroups.
Strong continuity on \(L^1(M)\) gives \(\varphi\in\operatorname{Dom}(\Delta_1)\),
with \(\Delta_1\varphi=\Delta\varphi\), and
\begin{equation}\label{eq:test-Bochner-L1-bound}
\|\varphi-P_t\varphi\|_{L^1(M)}
\le\min\{t\|\Delta\varphi\|_{L^1(M)},2\|\varphi\|_{L^1(M)}\}.
\end{equation}
The right-hand side is integrable against \(t^{-1-s}dt\). Phillips' theorem
\cite[Theorem 13.6]{schilling2012bernstein}, applied to the strongly continuous
contraction semigroup on \(L^1(M)\), therefore identifies the absolutely convergent
\(L^1\)-Bochner integral
\[
\frac{s}{\Gamma(1-s)}\int_0^\infty(\varphi-P_t\varphi)\,\frac{dt}{t^{1+s}}
=(-\Delta)_1^s\varphi=(-\Delta)^s\varphi.
\]
In particular, \(\varphi\in\operatorname{Dom}((-\Delta)_1^s)\); here
\(-(-\Delta)_1^s\) is the \(L^1\)-generator of \(T_t^{(s)}\).
This completes the proof.
\end{proof}

Now we prove the third characterization theorem.

\begin{proof}[\textbf{Proof of Theorem \ref{thm:conservative-characterization-zero-mean}.}]
We first prove that \textnormal{(i)} implies \textnormal{(ii)}. Fix $\varphi\in C_c^\infty(M)$, by Proposition \ref{prop:fractional-semigroup-Lp}, we can define
\[
F(t):=\int_M T_t^{(s)}\varphi\,dV_g,
\qquad t\ge0.
\]
Since $\varphi\in L^1(M)$ and the semigroup is conservative, Theorem~\ref{prop:conservativeness-subordination} yields
\[
F(t)=\int_M \varphi\,dV_g
\qquad\text{for all }t>0.
\]
Hence $F$ is constant on $(0,\infty)$, and therefore
\[
\lim_{t\downarrow0}\frac{F(t)-F(0)}{t}=0.
\]
On the other hand, by Lemma~\ref{lem:fraclap-test-L1}, Proposition~\ref{prop:fractional-semigroup-Lp}, and \cite[Section 6.1]{davies2007linear}, we have
\[
\frac{T_t^{(s)}\varphi-\varphi}{t}\to -(-\Delta)^s\varphi
\qquad\text{in }L^1(M)
\quad\text{as }t\downarrow0.
\]
Therefore,
\[
0
=
\lim_{t\downarrow0}\int_M \frac{T_t^{(s)}\varphi-\varphi}{t}\,dV_g
=
-\int_M (-\Delta)^s\varphi\,dV_g.
\]
This proves \textnormal{(ii)}.

We now prove that \textnormal{(ii)} implies \textnormal{(i)}. Suppose by contradiction that
\(\{T_t^{(s)}\}_{t\ge0}\) is not conservative. Then, by
Theorem~\ref{prop:conservativeness-subordination}, the heat semigroup
\(\{P_t\}_{t\ge0}\) is not conservative either. Hence there exist \(t_0>0\) and
\(x_0\in M\) such that $P_{t_0}\mathbf 1(x_0)<1.$
By continuity of \(x\mapsto P_{t_0}\mathbf 1(x)\), there exists a nonempty open set
\(U\Subset M\) such that $P_{t_0}\mathbf 1(x)<1$ for all $x\in U.$
Choose nontrivial \(\varphi\in C_c^\infty(M)\) such that $\varphi\ge0,\,\operatorname{supp}\varphi\subset U.$
Using the symmetry of the heat kernel, we obtain
\[
\int_M P_{t_0}\varphi\,dV_g
=
\int_M \varphi\,P_{t_0}\mathbf 1\,dV_g
<
\int_M \varphi\,dV_g.
\]
Define
\[
F(t):=\int_M P_t\varphi\,dV_g-\int_M \varphi\,dV_g,
\qquad t>0.
\]
Since \(\varphi\ge0\) and the heat semigroup is \(L^1(M)\)-contractive, we have
\[
\int_M P_t\varphi\,dV_g=\|P_t\varphi\|_{L^1(M)}
\le \|\varphi\|_{L^1(M)}
=\int_M \varphi\,dV_g,
\]
and hence
\(F(t)\le0 \quad\text{for all }t>0,\quad F(t_0)<0.\)
By the Bochner representation,
\[
(-\Delta)^s\varphi
=
\frac{s}{\Gamma(1-s)}
\int_0^\infty (\varphi-P_t\varphi)\,\frac{dt}{t^{1+s}}\in L^1(M).
\]
Integrating over \(M\), Fubini's theorem applies by \eqref{eq:test-Bochner-L1-bound} and gives
\[
\int_M (-\Delta)^s\varphi\,dV_g
=
-\frac{s}{\Gamma(1-s)}
\int_0^\infty F(t)\,\frac{dt}{t^{1+s}}.
\]
Since \(F(t)\le0\) for all \(t>0\), \(F(t_0)<0\), $F$ is continuous, it follows that
\[
\int_M (-\Delta)^s\varphi\,dV_g>0,
\]
which contradicts \textnormal{(ii)}. Therefore \(\{T_t^{(s)}\}_{t\ge0}\) must be conservative.
\end{proof}

The following simple observation shows that the heat-kernel part has zero mean. In particular, combined with the decomposition
\((-\Delta)^s\varphi = (-\Delta)_{\mathrm{hk}}^s\varphi+\mathcal R_s\,\varphi,\)
it immediately yields an explicit formula for $\int_M (-\Delta)^s\varphi\,dV_g.$

\begin{lemma}\label{lem:hk-zero-mean}
Let \((M,g)\) be a complete Riemannian manifold, let \(0<s<1\), and let
\(\varphi\in C_c^\infty(M)\). Then
\[
\int_M (-\Delta)_{\mathrm{hk}}^s\varphi\,dV_g=0.
\]
\end{lemma}

\begin{proof}
By definition,
\[
(-\Delta)_{\mathrm{hk}}^s\varphi(x)
=
\frac{s}{\Gamma(1-s)}
\int_0^\infty
\left(
\int_M \bigl(\varphi(x)-\varphi(y)\bigr)p(t,x,y)\,dV_g(y)
\right)\frac{dt}{t^{1+s}}.
\]
Set
\[
I_t(x):=\int_M \bigl(\varphi(x)-\varphi(y)\bigr)p(t,x,y)\,dV_g(y).
\]
Then $I_t(x)
=
\Theta(t,x)\varphi(x)-P_t\varphi(x),$
where
\[
\Theta(t,x):=\int_M p(t,x,y)\,dV_g(y).
\]

By the representation established in Section~\ref{Equiva}, the above formula is well defined.
To justify the following interchange, write
\(I_t=(\varphi-P_t\varphi)-(1-\Theta(t,\cdot))\varphi\).
Using \eqref{eq:test-Bochner-L1-bound} and the local boundedness of \(\mathcal R_s\), we obtain
\[
\begin{aligned}
\frac{s}{\Gamma(1-s)}\int_0^\infty\|I_t\|_{L^1(M)}\,\frac{dt}{t^{1+s}}
&\le\frac{s}{\Gamma(1-s)}
\left(\frac{\|\Delta\varphi\|_{L^1(M)}}{1-s}
+\frac{2\|\varphi\|_{L^1(M)}}s\right)\\
&\quad+\int_M\mathcal R_s|\varphi|\,dV_g<\infty.
\end{aligned}
\]
Thus \((-\Delta)_{\mathrm{hk}}^s\varphi\in L^1(M)\), and Fubini's theorem gives
\[
\int_M (-\Delta)_{\mathrm{hk}}^s\varphi\,dV_g
=
\frac{s}{\Gamma(1-s)}
\int_0^\infty
\left(
\int_M I_t(x)\,dV_g(x)
\right)\frac{dt}{t^{1+s}}.
\]

Now
\[
\int_M I_t(x)\,dV_g(x)
=
\int_M \Theta(t,x)\varphi(x)\,dV_g(x)
-
\int_M P_t\varphi(x)\,dV_g(x).
\]
Using Fubini's theorem and the symmetry of the heat kernel, we compute
\[
\begin{aligned}
\int_M P_t\varphi(x)\,dV_g(x)
&=
\int_M\int_M p(t,x,y)\varphi(y)\,dV_g(y)\,dV_g(x)\\
&=
\int_M\left(\int_M p(t,x,y)\,dV_g(x)\right)\varphi(y)\,dV_g(y)\\
&=
\int_M\left(\int_M p(t,y,x)\,dV_g(x)\right)\varphi(y)\,dV_g(y)\\
&=
\int_M \Theta(t,y)\varphi(y)\,dV_g(y).
\end{aligned}
\]
Hence $\int_M I_t(x)\,dV_g(x)=0.$
Substituting this into the previous identity, we obtain $\int_M (-\Delta)_{\mathrm{hk}}^s\varphi\,dV_g=0.$
This completes the proof.
\end{proof}

\begin{coro}\label{cor:int-fraclap-Rs}
Let \((M,g)\) be a complete Riemannian manifold, let \(0<s<1\), and let
\(\varphi\in C_c^\infty(M)\). Then
\[
\int_M (-\Delta)^s\varphi\,dV_g
=
\int_M \mathcal R_s(x)\,\varphi(x)\,dV_g(x).
\]
In particular, if \((M,g)\) is stochastically complete, then $\int_M (-\Delta)^s\varphi\,dV_g=0.$
\end{coro}

\begin{proof}
By \eqref{two equivalent},
\((-\Delta)^s\varphi = (-\Delta)_{\mathrm{hk}}^s\varphi+\mathcal R_s\,\varphi.\)
Integrating over \(M\) and using Lemma~\ref{lem:hk-zero-mean}, we obtain
\[
\int_M (-\Delta)^s\varphi\,dV_g
=
\int_M \mathcal R_s(x)\,\varphi(x)\,dV_g(x).
\]
If \((M,g)\) is stochastically complete, then \(\mathcal R_s\equiv0\), and the last
identity yields $\int_M (-\Delta)^s\varphi\,dV_g=0,$
which completes the proof.
\end{proof}

\subsection{The Operator Core Criterion}

In this subsection, we establish the fourth characterization of stochastic completeness via the \(L^1\)-core criterion. As an application, we also prove the uniqueness of bounded distributional solutions to the compensated elliptic equation associated with the fractional Laplacian.

\begin{proof}[\textbf{Proof of Theorem \ref{thm:L1-core-stochastic-completeness}.}]
Let \(\Delta_1\) be the heat-semigroup generator on \(L^1(M)\).
By Lemma~\ref{lem:fraclap-test-L1}, \(C_c^\infty(M)\subset\operatorname{Dom}(\Delta_1)\)
and \(\Delta_1\varphi=\Delta\varphi\) there. For \(\lambda>0\),
\(\lambda-\Delta_1\) is bijective from its graph-norm domain onto \(L^1(M)\), with
\[
\|h\|_{L^1(M)}\le\lambda^{-1}\|f\|_{L^1(M)},\qquad
\|\Delta_1h\|_{L^1(M)}\le2\|f\|_{L^1(M)},\qquad
h=(\lambda-\Delta_1)^{-1}f.
\]
Consequently, \(C_c^\infty(M)\) is a core for \(\Delta_1\) if and only if
\((\lambda-\Delta)C_c^\infty(M)\) is dense in \(L^1(M)\).

Assume that \((M,g)\) is stochastically complete. Fix \(\lambda>0\). We claim that
\(\mathcal R_\lambda:=\{(\lambda-\Delta)\varphi:\varphi\in C_c^\infty(M)\}\)
is dense in \(L^1(M)\). Suppose, by contradiction, that \(\mathcal R_\lambda\) is not dense in \(L^1(M)\). Then by the Hahn--Banach theorem and the duality \((L^1(M))^\ast=L^\infty(M)\), there exists a nonzero function $u\in L^\infty(M),\, u\not\equiv0,$
such that
\[
\int_M u\,(\lambda-\Delta)\varphi\,dV_g=0
\qquad\forall\,\varphi\in C_c^\infty(M).
\]
Thus \(u\) is a distributional solution of
\[
(\lambda-\Delta)u=0,
\qquad\text{that is,}\qquad \Delta u=\lambda u
\quad\text{in }M.
\]
By standard interior elliptic regularity, see \cite[Section 6.3]{grigoryan2009heat}, \(u\) is a bounded smooth solution of
\[
\Delta u=\lambda u
\qquad\text{in }M.\]
This contradicts the bounded, possibly signed, \(\lambda\)-harmonic characterization
\cite[Theorem 8.18 and Remark 8.19]{grigoryan2009heat}.
Indeed, bounded heat-equation uniqueness on every finite time interval gives
\(P_tu=e^{\lambda t}u\); contraction then implies
\(\|u\|_\infty\le e^{-\lambda t}\|u\|_\infty\) for all \(t>0\), so \(u=0\).
Hence \(\mathcal R_\lambda\) is dense in \(L^1(M)\). By the graph-norm equivalence above,
\(C_c^\infty(M)\) is a core for \(\Delta_1\).

Assume now that \(C_c^\infty(M)\) is a core for \(\Delta_1\). Fix \(\lambda>0\). Then by the same graph-norm equivalence, 
\(\mathcal R_\lambda=\{(\lambda-\Delta)\varphi:\varphi\in C_c^\infty(M)\}\)
is dense in \(L^1(M)\). We show that every bounded solution of $\Delta u=\lambda u$ in $M$
must vanish identically. Let \(u\in L^\infty(M)\) be such a bounded solution. By elliptic regularity, \(u\) is smooth. In particular, for every \(\varphi\in C_c^\infty(M)\), we have
\[
\int_M u\,(\lambda-\Delta)\varphi\,dV_g=0.
\]
Therefore
\[
\int_M u\,f\,dV_g=0
\qquad\forall\,f\in L^1(M).
\]
Hence \(u=0\) almost everywhere on \(M\), and by smoothness, \(u\equiv0\).

We have shown that the equation $\Delta u=\lambda u$
admits no nontrivial bounded solution. By  
\cite[Theorem 8.18]{grigoryan2009heat}, \((M,g)\) is stochastically complete. This completes the proof.
\end{proof}

\begin{lemma}\label{lem:fractional-L1-core}
Let \(0<s<1\). If \(M\) is stochastically complete, then
\(C_c^\infty(M)\) is a core for the \(L^1\)-generator \(-(-\Delta)_1^s\).
\end{lemma}

\begin{proof}
Let \(\Delta_1\) be the heat-semigroup generator on \(L^1(M)\).
For every complete manifold and \(h\in\operatorname{Dom}(\Delta_1)\),
the semigroup identity and contraction give
\(\|h-P_th\|_1\le\min\{t\|\Delta_1h\|_1,2\|h\|_1\}\).
Phillips' formula, split at \(t=1\), therefore yields
\begin{equation}\label{eq:fractional-generator-graph-bound}
\|(-\Delta)_1^sh\|_1
\le\frac{s}{\Gamma(1-s)}
\left(\frac{\|\Delta_1h\|_1}{1-s}+\frac{2\|h\|_1}{s}\right).
\end{equation}
By \cite[Proposition 13.5]{schilling2012bernstein}, \(\operatorname{Dom}(\Delta_1)\)
is a core for \((-\Delta)_1^s\); the result applies to the strongly continuous
contraction semigroup \(P_t\) and the Bernstein function \(\lambda^s\).
If \(M\) is stochastically complete, Theorem~\ref{thm:L1-core-stochastic-completeness}
approximates each \(h\in\operatorname{Dom}(\Delta_1)\) by \(C_c^\infty(M)\)
in the \(\Delta_1\)-graph norm. Applying \eqref{eq:fractional-generator-graph-bound}
to the differences gives convergence also in the \((-\Delta)_1^s\)-graph norm.
Combining the two approximations proves the assertion.
\end{proof}

 We now show that stochastic completeness is equivalent to the uniqueness of the trivial bounded distributional solution to the compensated elliptic equation.

\begin{proof}[\textbf{Proof of Theorem \ref{prop:unique-trivial-solution-equivalence}.}]
We prove the equivalence by showing \((i)\Rightarrow(ii)\) and \((ii)\Rightarrow(i)\).

\smallskip

\noindent
\((i)\Rightarrow(ii)\).
Assume that \((M,g)\) is stochastically complete. Then, by \eqref{two equivalent},
\[
(-\Delta)^s\varphi = (-\Delta)_{\mathrm{hk}}^s\varphi
\qquad\forall\,\varphi\in C_c^\infty(M),
\]
and $\mathcal R_s\equiv 0.$
Thus the equation
\[
(-\Delta)^s v+\alpha v=\mathcal R_s
\qquad\text{in }M
\]
reduces to
\[
(-\Delta)^s v+\alpha v=0
\qquad\text{in }M.
\]

Let \(\psi\in C_c^\infty(M)\), and set $\eta:=R_\alpha^{(s)}\psi.$
Then \(\eta\in \operatorname{Dom}((-\Delta)_1^s)\), where \(-(-\Delta)_1^s\) is the
\(L^1\)-generator of the subordinate semigroup, and by the definition of the \(L^1\)-resolvent,
\[
\bigl((-\Delta)^s+\alpha\bigr)\eta=\psi
\qquad\text{in }L^1(M).
\]
By Lemma~\ref{lem:fractional-L1-core}, \(C_c^\infty(M)\) is a core for the \(L^1\)-generator
\(-(-\Delta)_1^s\). In the following limits \((-\Delta)^s\) denotes this compatible
\(L^1\)-realization. Therefore, there exists a sequence \(\{\eta_n\}\subset C_c^\infty(M)\) such that
\[
\eta_n\to\eta
\quad\text{in }L^1(M),
\qquad
(-\Delta)^s\eta_n\to(-\Delta)^s\eta
\quad\text{in }L^1(M).
\]
Hence,
\[
(-\Delta)^s\eta_n+\alpha\eta_n\to\psi
\qquad\text{in }L^1(M).
\]

Since \(v\in L^\infty(M)\) is a distributional solution of
\[
(-\Delta)^s v+\alpha v=0
\qquad\text{in }M,
\]
testing against \(\eta_n\in C_c^\infty(M)\) gives
\[
\int_M v\,\bigl((-\Delta)^s\eta_n+\alpha\eta_n\bigr)\,dV_g=0
\qquad\text{for every }n.
\]
Passing to the limit as \(n\to\infty\), using that \(v\in L^\infty(M)\) and
\[
(-\Delta)^s\eta_n+\alpha\eta_n\to\psi
\qquad\text{in }L^1(M),
\]
we obtain $\int_M v\,\psi\,dV_g=0.$
Since \(\psi\in C_c^\infty(M)\) was arbitrary, it follows that \(v=0\) in the sense of distributions, and hence \(v=0\) almost everywhere on \(M\).

\smallskip

\noindent
\((ii)\Rightarrow(i)\).
Assume now that \(M\) is not stochastically complete. Consider the defect function
\(v_\alpha:=\mathbf 1-\alpha R_\alpha^{(s)}\mathbf 1.\)
Since $0\le \alpha R_\alpha^{(s)}\mathbf 1\le \mathbf 1,$
we have \(v_\alpha\in L^\infty(M)\) and \(v_\alpha\ge0\). Moreover, by Theorem~\ref{thm:fractional-conservative-equivalence}, \(v_\alpha\not\equiv0\).

On the other hand, by Proposition~\ref{solutionrs}, for every \(\varphi\in C_c^\infty(M)\),
\[
\int_M v_\alpha\bigl((-\Delta)^s\varphi+\alpha\varphi\bigr)\,dV_g
=
\int_M (-\Delta)^s\varphi\,dV_g.
\]
By Corollary~\ref{cor:int-fraclap-Rs},
\[
\int_M (-\Delta)^s\varphi\,dV_g
=
\int_M \mathcal R_s(x)\varphi(x)\,dV_g(x).
\]
Therefore,
\[
\int_M v_\alpha\bigl((-\Delta)^s\varphi+\alpha\varphi\bigr)\,dV_g
=
\int_M \mathcal R_s(x)\varphi(x)\,dV_g(x)
\qquad
\forall\,\varphi\in C_c^\infty(M).
\]
That is, \(v_\alpha\) is a bounded distributional solution of
\[
(-\Delta)^s v+\alpha v=\mathcal R_s
\qquad\text{in }M.
\]
Since \(v_\alpha\not\equiv0\), this contradicts the uniqueness of the trivial bounded distributional solution. Hence \(M\) must be stochastically complete.
\end{proof}

\subsection{Elliptic and Parabolic Characterizations}

We first show that the fractional resolvent provides a bounded
distributional solution of the corresponding elliptic equation.
We then establish the equivalence between conservativeness and
uniqueness of bounded distributional solutions to the fractional
elliptic and parabolic equations.

\begin{proposition}\label{solutionrs}
	Let $(M,g)$ be a complete Riemannian manifold, $0<s<1$,
	$\alpha>0$, and $f\in L^\infty(M)$. Then
	$u_\alpha:=R_\alpha^{(s)}f$ is a bounded distributional solution of
	\[
	(-\Delta)^su+\alpha u=f
	\qquad\text{in }M,
	\]
	with $\|u_\alpha\|_\infty\le\alpha^{-1}\|f\|_\infty$.
\end{proposition}

\begin{proof}
	By the $L^\infty$-contractivity of $T_t^{(s)}$,
	\[
	\|u_\alpha\|_\infty
	\le\int_0^\infty e^{-\alpha t}\|T_t^{(s)}f\|_\infty\,dt
	\le\alpha^{-1}\|f\|_\infty.
	\]
	Fix $\varphi\in C_c^\infty(M)$ and set
	$\psi:=(-\Delta)^s\varphi+\alpha\varphi$.
	By Lemma~\ref{lem:fraclap-test-L1},
	$\varphi\in\operatorname{Dom}(A_s)$ and
	$A_s\varphi=(-\Delta)^s\varphi\in L^1(M)$,
	where $A_s:=(-\Delta)_1^s$ is the positive generator
	on $L^1(M)$. In particular, $\psi\in L^1(M)$.
	
	The symmetry of the subordinate heat kernel and Fubini's theorem give
	\[
	\langle u_\alpha,\psi\rangle
	=\int_0^\infty e^{-\alpha t}
	\langle f,T_t^{(s)}\psi\rangle\,dt.
	\]
	Since $\varphi\in\operatorname{Dom}(A_s)$, semigroup theory gives
	$T_t^{(s)}\psi=-\partial_tT_t^{(s)}\varphi
	+\alpha T_t^{(s)}\varphi$ in $L^1(M)$.
	Setting $F(t):=\langle f,T_t^{(s)}\varphi\rangle$, we obtain
	\[
	\begin{aligned}
		\langle u_\alpha,\psi\rangle
		&=\int_0^\infty e^{-\alpha t}
		\bigl(-F'(t)+\alpha F(t)\bigr)\,dt=-\int_0^\infty
		\frac{d}{dt}\bigl(e^{-\alpha t}F(t)\bigr)\,dt
		=\langle f,\varphi\rangle.
	\end{aligned}
	\]
	Here the boundary term at infinity vanishes since
	$|F(t)|\le\|f\|_\infty\|\varphi\|_1$, while
	$F(t)\to\langle f,\varphi\rangle$ as $t\downarrow0$
	by strong continuity on $L^1(M)$.
	Consequently,
	\[
	\int_M u_\alpha(-\Delta)^s\varphi\,dV_g
	+\alpha\int_M u_\alpha\varphi\,dV_g
	=\int_M f\varphi\,dV_g,
	\qquad \forall\,\varphi\in C_c^\infty(M).
	\]
	This completes the proof.
\end{proof}

The following lemma justifies the time-dependent test functions
used in the parabolic uniqueness argument. We write
$\langle u,h\rangle:=\int_M uh\,dV_g$ for the duality
between $L^\infty(M)$ and $L^1(M)$.

\begin{lemma}\label{lem:parabolic-generator-duality}
	Let $(M,g)$ be stochastically complete, $0<s<1$, and
	$A_s:=(-\Delta)_1^s$. Fix $T>0$, and suppose that
	$u\in L^\infty(M\times(0,T))$ is a distributional solution of $\partial_tu+(-\Delta)^su=0$
	with initial datum $u_0\in L^\infty(M)$.
	Then $u$ admits a weak-$*$ continuous representative
	$\widetilde u:[0,T]\to L^\infty(M)$ satisfying
	$\widetilde u(0)=u_0$.
	
	Moreover, for $0\le a<b\le T$ and
	$\eta\in C^1([a,b];\operatorname{Dom}(A_s))$,
	where the domain is equipped with the graph norm,
	\begin{equation}\label{eq:parabolic-endpoint-duality}
		\langle\widetilde u(b),\eta(b)\rangle
		-\langle\widetilde u(a),\eta(a)\rangle
		=
		\int_a^b
		\langle\widetilde u(t),\eta'(t)-A_s\eta(t)\rangle\,dt.
	\end{equation}
\end{lemma}

\begin{proof}
	Set $K:=\|u\|_{L^\infty(M\times(0,T))}$.
	For $h\in\operatorname{Dom}(A_s)$, the core property
	in Lemma~\ref{lem:fractional-L1-core} gives
	$h_j\in C_c^\infty(M)$ such that
	$h_j\to h$ and $A_sh_j\to A_sh$ in $L^1(M)$.
	Testing the distributional equation with
	$\chi(t)h_j(x)$, where $\chi\in C_c^\infty([0,T))$,
	and passing to the limit yields
	\begin{equation*}
		\int_0^T
		\bigl(-\chi'(t)\langle u(t),h\rangle
		+\chi(t)\langle u(t),A_sh\rangle\bigr)\,dt
		=\chi(0)\langle u_0,h\rangle.
	\end{equation*}
	The passage to the limit follows from
	$\|u(t)\|_\infty\le K$ almost everywhere.
	Thus $t\mapsto\langle u(t),h\rangle$ has the
	absolutely continuous representative
	\begin{equation}\label{eq:parabolic-scalar-representative}
		y_h(t)=\langle u_0,h\rangle
		-\int_0^t\langle u(r),A_sh\rangle\,dr,
		\qquad 0\le t\le T.
	\end{equation}
	Since $|y_h(t)|\le K\|h\|_1$ almost everywhere,
	continuity extends this bound to every $t\in[0,T]$.
	
	For each $t$, the map $h\mapsto y_h(t)$ is therefore
	a bounded linear functional on $\operatorname{Dom}(A_s)$
	with respect to the $L^1$-norm. By density, it extends
	to $L^1(M)$ and defines $\widetilde u(t)\in L^\infty(M)$,
	with $\|\widetilde u(t)\|_\infty\le K$.
	Using a countable dense subset of $\operatorname{Dom}(A_s)$,
	we obtain $\widetilde u(t)=u(t)$ for almost every $t$.
	Formula \eqref{eq:parabolic-scalar-representative} and density
	also give weak-$*$ continuity and $\widetilde u(0)=u_0$.
	
	For $\eta(t)=\sum_{j=1}^N c_j(t)h_j$, with
	$c_j\in C^1([a,b])$ and $h_j\in\operatorname{Dom}(A_s)$,
	the scalar product rule applied to $c_jy_{h_j}$
	gives \eqref{eq:parabolic-endpoint-duality}.
	For general $\eta\in C^1([a,b];\operatorname{Dom}(A_s))$,
	approximate $\eta'$ uniformly in the graph norm by
	finite linear combinations of fixed domain elements
	with continuous scalar coefficients, and integrate from $a$.
	This gives finite-rank functions converging to $\eta$
	in the $C^1$ graph norm. Since
	$\|\widetilde u(t)\|_\infty\le K$, passage to the limit
	proves \eqref{eq:parabolic-endpoint-duality}.
\end{proof}

We now prove the equivalence between conservativeness and the
elliptic and parabolic uniqueness properties.

\begin{proof}[\textbf{Proof of Theorem \ref{maintheorem}.}]
	We prove the four implications separately.
	
	\medskip
	\noindent\textbf{Step 1. $(i)\Rightarrow(ii)$.}
	Suppose that $T_t^{(s)}\mathbf1=\mathbf1$ for every $t>0$.
	By Theorem~\ref{thm:conservative-characterization-zero-mean},
	the constant $\alpha^{-1}$ is a bounded distributional solution of
	$(-\Delta)^su+\alpha u=\mathbf1$.
	
	If $u\in L^\infty(M)$ is another bounded distributional solution,
	then $v:=u-\alpha^{-1}$ satisfies
	\[
	\int_M v(-\Delta)^s\varphi\,dV_g
	+\alpha\int_M v\varphi\,dV_g=0,
	\qquad\forall\,\varphi\in C_c^\infty(M).
	\]
	By Theorem~\ref{prop:unique-trivial-solution-equivalence},
	$v=0$ almost everywhere on $M$. Hence $u=\alpha^{-1}$.
	
	\medskip
	\noindent\textbf{Step 2. $(ii)\Rightarrow(i)$.}
	Assume that $\alpha^{-1}$ is the unique bounded distributional
	solution of $(-\Delta)^su+\alpha u=\mathbf1$.
	By Proposition~\ref{solutionrs}, $R_\alpha^{(s)}\mathbf1$
	is also a bounded distributional solution. Therefore,
	$R_\alpha^{(s)}\mathbf1=\alpha^{-1}$ almost everywhere.
	Theorem~\ref{thm:fractional-conservative-equivalence}
	then gives $T_t^{(s)}\mathbf1=\mathbf1$ for every $t>0$.
	
	\medskip
	\noindent\textbf{Step 3. $(i)\Rightarrow(iii)$.}
	Assume that $T_t^{(s)}$ is conservative. Let $w$ be a
	distributional solution of
	\[
	\partial_t w+(-\Delta)^s w=0,\qquad w(\cdot,0)=\mathbf1,
	\]
	such that $w\in L^\infty(M\times(0,T))$ for every $T>0$.
	By Theorem~\ref{thm:conservative-characterization-zero-mean},
	the constant function $\mathbf1$ satisfies the same equation.
	Hence $z:=w-\mathbf1$ is a distributional solution with zero
	initial datum.
	
	Fix $T>0$ and put $A_s:=(-\Delta)_1^s$.
	By Theorem~\ref{prop:conservativeness-subordination},
	$M$ is stochastically complete. Thus
	Lemma~\ref{lem:parabolic-generator-duality} gives a weak-$*$
	continuous representative $\widetilde z$ on $[0,T]$
	with $\widetilde z(0)=0$.
	
	For $0<t_0<T$ and $h\in\operatorname{Dom}(A_s^2)$,
	choose the backward test function
	\[
	\eta(t):=T_{t_0-t}^{(s)}h,\qquad 0\le t\le t_0.
	\]
	Since the semigroup preserves $\operatorname{Dom}(A_s^2)$
	and commutes with $A_s$, we have
	$\eta\in C^1([0,t_0];\operatorname{Dom}(A_s))$
	in the graph norm, with $\eta'=A_s\eta$.
	Applying \eqref{eq:parabolic-endpoint-duality}, we obtain
	\[
	\langle\widetilde z(t_0),h\rangle
	-\langle\widetilde z(0),T_{t_0}^{(s)}h\rangle
	=\int_0^{t_0}
	\langle\widetilde z(t),\eta'(t)-A_s\eta(t)\rangle\,dt=0.
	\]
	Since $\widetilde z(0)=0$, it follows that
	$\langle\widetilde z(t_0),h\rangle=0$.
	
	To conclude, $\operatorname{Dom}(A_s^2)$ is dense in $L^1(M)$.
	Indeed, set $Q_n:=nR_n^{(s)}$. Strong continuity of the semigroup
	gives $Q_nf\to f$ in $L^1(M)$, and
	\[
	Q_n^2f\in\operatorname{Dom}(A_s^2),\qquad
	\|Q_n^2f-f\|_1\le2\|Q_nf-f\|_1\longrightarrow0.
	\]
	Consequently, $\widetilde z(t_0)=0$ in $L^\infty(M)$.
	As $t_0$ and $T$ are arbitrary, $w=\mathbf1$ almost everywhere
	on $M\times(0,\infty)$. This proves $(iii)$.
	
	\medskip
	\noindent\textbf{Step 4. $(iii)\Rightarrow(i)$.}
	Assume that the only bounded distributional solution of
	\[
	\partial_t w+(-\Delta)^s w=0,\qquad w(\cdot,0)=\mathbf1,
	\]
	is $w\equiv\mathbf1$. We show that
	$w(x,t):=T_t^{(s)}\mathbf1(x)$ is such a solution.
	
	Let $\Phi\in C_c^\infty(M\times[0,\infty))$ and
	$A_s:=(-\Delta)_1^s$. By Lemma~\ref{lem:fraclap-test-L1},
	$A_s\Phi(\cdot,t)=(-\Delta)^s\Phi(\cdot,t)$.
	Moreover, all spatial supports lie in one compact set,
	and \eqref{eq:fractional-generator-graph-bound} implies that
	$t\mapsto\Phi(\cdot,t)$ is continuously differentiable
	in the graph norm of $A_s$.
	
	The semigroup product rule therefore gives, in $L^1(M)$,
	\[
	T_t^{(s)}\bigl(-\partial_t\Phi(\cdot,t)+A_s\Phi(\cdot,t)\bigr)
	=-\frac{d}{dt}\bigl(T_t^{(s)}\Phi(\cdot,t)\bigr).
	\]
	By symmetry of the subordinate heat kernel,
	$\langle T_t^{(s)}\mathbf1,h\rangle
	=\langle\mathbf1,T_t^{(s)}h\rangle$ for every $h\in L^1(M)$.
	Hence
	\[
	\begin{aligned}
		&\int_0^\infty\int_M T_t^{(s)}\mathbf1(x)
		\bigl(-\partial_t\Phi(x,t)+(-\Delta)^s\Phi(x,t)\bigr)
		\,dV_g(x)\,dt\\
		&\qquad=-\int_0^\infty
		\frac{d}{dt}\left(\int_M T_t^{(s)}
		\Phi(\cdot,t)\,dV_g\right)dt
		=\int_M\Phi(x,0)\,dV_g(x).
	\end{aligned}
	\]
	The last equality follows from the compact time support of
	$\Phi$ and the strong continuity of $T_t^{(s)}$ on $L^1(M)$. Thus $T_t^{(s)}\mathbf1$ is a bounded distributional solution
	with initial datum $\mathbf1$. By uniqueness,
	$T_t^{(s)}\mathbf1=\mathbf1$ almost everywhere on
	$M\times(0,\infty)$.
	
	For every $h\in L^1(M)$, the map
	$t\mapsto\langle T_t^{(s)}\mathbf1,h\rangle
	=\int_M T_t^{(s)}h\,dV_g$
	is continuous. Therefore,
	$T_t^{(s)}\mathbf1=\mathbf1$ in $L^\infty(M)$
	for every $t>0$. By spatial continuity of
	$T_t^{(s)}\mathbf1$, the identity holds pointwise on $M$.
	Hence the subordinate semigroup is conservative.
	This completes the proof.
\end{proof}

% ================= Acknowledgements =================

\Needspace{7\baselineskip}
\noindent\textbf{Conflict of interest.} The authors declare that they have no conflict of interest.

\medskip

\noindent\textbf{Acknowledgements.} The authors are sincerely grateful to Professor Alexander Grigor'yan for his valuable suggestions and helpful guidance concerning this work. They also thank Professor Daniel Hauer for his helpful comments, advice, and encouragement.

% ================= Bibliography =================
	\printbibliography

\bigskip

\Needspace{7\baselineskip}
\noindent\textit{Rui Chen}: School of Mathematical Sciences, Fudan University,\\[1mm]
Shanghai 200433, China\\[1mm]
Brandenburg University of Technology Cottbus--Senftenberg,\\[1mm]
Cottbus 03046, Germany\\[1mm]
\noindent\emph{Email:} \texttt{chenrui23@m.fudan.edu.cn}

\vspace{1em}

\Needspace{7\baselineskip}
\noindent\textit{Bobo Hua}: School of Mathematical Sciences, Fudan University,\\[1mm]
Shanghai 200433, PR China\\[1mm]
Shanghai Center for Mathematical Sciences, Fudan University,\\[1mm]
Shanghai 200433, PR China\\[1mm]
\noindent\emph{Email:} \texttt{bobohua@fudan.edu.cn}
\end{document}